\documentclass[11pt]{amsart}

\usepackage{epigamath}

\usepackage[english]{babel}

\numberwithin{equation}{section}

\usepackage{enumitem}
\setlist[enumerate,1]{label={\rm(\roman*)}, ref={\rm\roman*}} 

\usepackage{extpfeil}
\usepackage[all]{xy}
\usepackage{amssymb} 
\usepackage{euscript}
\usepackage{mathtools}
\usepackage{amsthm}
\usepackage{verbatim}
\usepackage{mathrsfs}
\usepackage{url}
\usepackage{color}

\usepackage{tikz-cd}

\theoremstyle{plain}
	\newtheorem{thm}{Theorem}[section]
	\newtheorem{prop}[thm]{Proposition}
	\newtheorem{lem}[thm]{Lemma}

	\newtheorem{cor}[thm]{Corollary}
	\newtheorem{conj}[thm]{Conjecture}
	\newtheorem{lem/defn}[thm]{Lemma/Definition}

\theoremstyle{definition}
	\newtheorem{defn}[thm]{Definition}
	\newtheorem{notation}[thm]{Notation}
	\newtheorem{const}[thm]{Construction}
    
\theoremstyle{remark}
	\newtheorem{rem}[thm]{Remark}
	\newtheorem{ex}[thm]{Example}
	
        \newtheorem{warn}[thm]{Warning}

\def\nc{\newcommand}
\def\on{\operatorname}

\nc{\edit}[1]{\marginpar{\footnotesize{#1}}}

\nc{\C}{\mathbb{C}}
\nc{\Z}{\mathbb{Z}}
\nc{\PP}{\mathbb{P}}
\nc{\R}{\mathbb{R}}
\nc{\F}{\mathbb{F}}
\nc{\AAA}{\mathbb{A}}
\nc{\LL}{\mathbb{L}}
\nc{\OO}{\mathcal{O}}

\nc{\X}{\EuScript{X}}
\nc{\sZ}{\EuScript{Z}}

\nc{\id}{{\on{id}}}
\nc\Hom{{\on{Hom}}}
\nc\cone{{\on{cone}}}
\nc{\Rep}{{\on{Rep}}}
\nc\Ob{{\on{Ob}}}
\nc\Spec{{\on{Spec}}}
\nc\Mod{{\on{Mod}}}
\nc\coMod{{\on{coMod}}}
\nc\Perf{{\on{Perf}}}
\nc\End{{\on{End}}}
\nc{\into}{\hookrightarrow}
\nc{\tr}{\on{tr}}
\nc{\ev}{\on{ev}}
\nc{\im}{\on{im}}
\nc{\Mot}{\on{Mot}}
\nc{\filstack}{\AAA^1 / \GG_m}
\nc{\pt}{\on{pt}}
\nc{\formalgroup}{\widehat{\GG}}
\nc{\coker}{\on{coker}}
\nc{\rk}{\on{rank}}
\nc{\TOP}{\on{Top}_{\mathbb{C}}^{s}}
\nc{\gr}{\on{gr}}
\nc{\Catperf}{\text{Cat}^{\text{perf}}}
\nc{\Sym}{\on{Sym}}
\nc{\xra}{\xrightarrow}
\nc{\Bet}{\mathbf{Betti}_{X}}
\nc{\spectralift}{D(\formalgroup^{un})}
\nc{\Fix}{\mathsf{Fix}}
\nc{\codim}{\on{codim}}
\nc{\Fred}{\on{Fred}}
\nc{\colim}{\on{colim}}
\nc{\KK}{{\bf K}}
\nc{\Sp}{\on{Sp}}
\nc{\onto}{\twoheadrightarrow}
\nc{\A}{\mathbb{A}}
\nc{\Aff}{\on{Aff}}
\nc{\SH}{\on{SH}}
\nc{\QCoh}{\on{QCoh}}
\nc{\Alg}{\on{Alg}}
\nc{\Br}{\on{Br}}
\nc{\ta}{\widetilde{\a}}
\nc{\Shv}{\on{Shv}}
\nc{\GG}{\mathbb{G}}
\nc{\an}{\on{an}}

\nc{\Pre}{\on{Pre}}
\nc{\qc}{\on{qc}}
\nc{\op}{\on{op}}
\nc{\shEnd}{{\mathcal End}}
\nc{\Sph}{\mathbb{S}}

\nc{\Map}{\on{Map}}
\nc{\Vect}{\on{Vect}}
\nc{\holim}{\on{holim}}

\def\H{\mathcal{H}}

\def\A{\mathcal{A}}

\def\a{\alpha}

\def\Perf{\on{Perf}}

\def\Sp{\on{Sp}}
\def\fib{\on{fib}}

\nc{\W}{\mathbb{W}}

\def\QCoh{\on{QCoh}}

\nc{\lra}{\longrightarrow}

\nc\Def{{\on{Def}}}
\nc\Grp{{\on{Grp}}}
\nc\fil{{\on{fil}}}
\nc\dStk{{\on{dStk}}}
\nc\sStk{{\on{sStk}}}
\nc\SStk{{\on{SStk}}}
\nc\Stk{{\on{Stk}}}
\nc\KU{{\on{KU}}}

\nc\un{{\mathrm{un}}}
\nc\ad{{\mathrm{ad}}}
\nc\sm{{\mathrm{sm}}}
\nc\red{{\mathrm{red}}}
\nc\et{{\mathrm{\acute{e}t}}}
\nc\gp{{\mathrm{gp}}}
\nc\cn{{\mathrm{cn}}}
\nc\fd{{\mathrm{fd}}}
\nc\cl{{\mathrm{cl}}}

\newcommand{\supth}[1]{\ensuremath{#1^{\mathrm{th}}}}

\EpigaVolumeYear{8}{2024} \EpigaArticleNr{2} \ReceivedOn{July 1, 2021}
\InFinalFormOn{September 27, 2023}
\AcceptedOn{October 21, 2023}

\title{Filtered formal groups, Cartier duality, \\ and derived algebraic geometry}
\titlemark{Filtered formal groups, Cartier duality, and derived algebraic geometry}

\author{Tasos Moulinos}
\address{Institute of Advanced Study,
1 Einstein Drive,
Princeton, NJ 08540, USA }
\email{tmoulinos@gmail.com}

\authormark{T.~Moulinos}

\AbstractInEnglish{We develop a notion of formal groups in the filtered setting and describe a duality relating these to a specified class of filtered Hopf algebras. We then study a deformation to the normal cone construction in the setting of derived algebraic geometry. Applied to the unit section of a formal group $\formalgroup$, this provides a $\GG_m$-equivariant degeneration of $\formalgroup$ to its tangent Lie algebra. We prove a unicity result on complete filtrations, which, in particular, identifies the resulting filtration on the coordinate algebra of this deformation with the adic filtration on the coordinate algebra of $\formalgroup$. We use this in a special case, together with the aforementioned notion of Cartier duality, to recover the filtration on the filtered circle introduced by the author, Robalo, and To{\"e}n in a 2019 paper.
  Finally, we investigate some properties of $\formalgroup$-Hochschild homology, set out in that paper  
  and describe ``lifts'' of these invariants to the setting of spectral algebraic geometry.}

  

\acknowledgement{This work is supported by the grant NEDAG ERC-2016-ADG-741501.}

\begin{document}


\maketitle

\begin{prelims}

\DisplayAbstractInEnglish

\bigskip

\DisplayKeyWords

\medskip

\DisplayMSCclass

\end{prelims}


\newpage

\setcounter{tocdepth}{1}

\tableofcontents


\section{Introduction}\label{sec1}
The starting point of this work arises from the construction in  \cite{moulinos2019universal} of the \emph{filtered circle}, an object of algebro-geometric nature, capturing the $k$-linear homotopy type of $S^1$, the topological circle. This construction is motivated by the schematization problem due to Grothendieck, stated most generally in finding a purely algebraic description of the $\Z$-linear homotopy type of an arbitrary topological space $X$. 

In the process of doing this, the authors realized that there was an inextricable link between this construction and the theory of formal groups and Cartier duality, as set out in \cite{cartier1962groupes}. 
We briefly review the relationship. The filtered circle is obtained as the classifying stack $B \mathbb{H} $,
where $\mathbb{H}$ is a $\GG_m$-equivariant family of group schemes parametrized by the affine line, $\AAA^1$. This family of schemes interpolates between two affine group schemes, $\mathsf{Fix}$ and $\mathsf{Ker}$; these can be traced to the work of \cite{sekiguchi2001note}, where they are shown to arise via Cartier duality from the formal multiplicative and formal additive groups, $\widehat{\GG_m}$ and $\widehat{\GG_a}$, respectively. The filtered circle $S^1_{\fil}$
is then obtained as $B \mathbb{H}$, the classifying stack over $\filstack$ of $\mathbb{H}$. By taking the derived mapping space out of $S^1_{\fil}$ in $\filstack$-parametrized derived stacks, one recovers precisely Hochshild homology together with a functorial filtration. 

There is no reason to stop at $\widehat{\GG_m}$ or $\widehat{\GG_a}$, however. In \cite{sekiguchi2001note}, the authors proposed, given an arbitrary $1$-dimensional formal group $\formalgroup$, the following generalized notion of Hochshild homology of simplicial commutative rings: 
$$
\on{HH}^{\formalgroup}(-)\colon \on{sCAlg}_k \lra \on{sCAlg}_k, \quad A \longmapsto \on{HH}^{\formalgroup}(A) := R \Gamma ( \Map_{\dStk_k}(B \formalgroup^\vee, \Spec A )).
$$
The right-hand side is  the derived mapping space out of $B \formalgroup^\vee$, the classifying stack of the Cartier dual of $\formalgroup$. 
For $\formalgroup= \widehat{\GG_m}$ one recovers Hochshild homology via a  natural equivalence of derived schemes  
$$
\Map(B \mathsf{Fix}, X) \lra \Map(S^1, X),  
$$
and for $\formalgroup= \widehat{\GG_a}$ one recovers the derived de Rham  algebra (\textit{cf}. \cite{toen2011derham}) via an equivalence  
$$
\Map(B \mathsf{Ker}, X) \simeq \mathbb{T}_{X|k}[-1] = \Spec(\on{Sym}(\mathbb{L}_{X|k}[1])
$$
with the shifted (negative) tangent bundle. One may now ask the following natural questions: if one replaces $\widehat{\GG_m}$ with  an arbitrary formal group $\formalgroup$, does one obtain a  degeneration over $\filstack$ similar to that in the previous paragraph? Is there a sense in which such a degeneration is canonical?

The overarching aim of this paper is to address some of these questions by further systematizing some of the above ideas, particularly using further ideas from spectral and derived algebraic geometry.

\subsection{Filtered formal groups}
The first main undertaking of this paper is to introduce a notion of \emph{filtered formal group} over a (discrete) ring $R$. For now, we give the following rough definition, postponing the full definition to Section~\ref{filteredformalgroupsection}. 

\begin{defn}[\textit{cf}.\ Definition~\ref{cogroupdefinition}] 
Let $R$ be a discrete commutative ring. A \emph{filtered formal group} is an abelian cogroup object $A$ in the category of complete filtered algebras $\on{CAlg}(\widehat{\on{Fil}}_R)$ (see Definition~\ref{completenessfiltered} for an explanation of the notation $\widehat{\on{Fil}_R}$) which are discrete at the level of underlying algebras. 
\end{defn}

Heuristically, these give rise to stacks 
$$
\formalgroup \lra \filstack, 
$$
for which the pullback $\pi^*(\formalgroup)$ along  the smooth atlas $\pi\colon \AAA^1 \to \filstack$ is a formal group over $\AAA^1$ in the classical sense. 

From the outset we restrict to a full subcategory of complete filtered algebras, for which there exists a well-behaved duality theory.  Our setup
is inspired by the framework  of \cite{ellipticII} and the notion of smooth coalgebra therein. Namely, we restrict to complete filtered algebras that arise as the duals of \emph{smooth filtered coalgebras} (\textit{cf}.\ Definition~\ref{smoothfilteredcoalg}). The abelian cogroup structure on a complete filtered algebra $A$ then corresponds to the structure of an abelian group object on the corresponding coalgebra. As everything in sight is discrete, hence $1$-categorical (\textit{cf}.\ Remark~\ref{abelianobjects1cat}), this is precisely the data of a comonoid in smooth coalgebras, \textit{i.e.}, a filtered Hopf algebra. 
Inspired by the classical Cartier duality correspondence over a field between formal groups and affine group schemes, we refer to this as filtered Cartier duality.

\begin{rem}
We acknowledge that the phrase ``Cartier duality'' has a variety of different meanings throughout the literature (\textit{e.g.}, duality between finite group schemes, $p$-divisible groups, \textit{etc}.) For us, this will always mean a contravariant correspondence between (certain full subcategories of) formal groups and affine group schemes, originally observed by Cartier over a field in \cite{cartier1962groupes}. 
\end{rem}

\begin{rem}
In this paper we are concerned with filtered formal groups $\formalgroup \to \filstack$ whose ``fiber over $\Spec(R) \to \filstack$''  recovers a classical (discrete) formal group. We conjecture that the duality theory of Section~\ref{filteredformalgroupsection} holds true in the filtered, spectral setting. Nevertheless, as this takes us away from our main applications, we have stayed away from this level of generality.  
\end{rem}
As it turns out, the notion of a complete filtered algebra, and hence ultimately the notion of a filtered formal group, is of a rigid nature. To this effect, we demonstrate the following unicity result on complete filtered algebras $A_*$ with a specified associated graded (before taking any group structure into account). In order to state this, we recall that given any commutative ring $A$ with ideal $I$, there exists a filtered object $F_I^{*}(A)$, the \emph{adic filtration} on $A$. 

\begin{thm}\label{introadicunicity} 
Let $A$ be a commutative ring which is complete with respect to the $I$-adic topology induced by some ideal $I \subset A$. Let $A_* \in \on{CAlg}(\widehat{\on{Fil}}_R)$ be a $($discrete$)$ complete filtered algebra with underlying object $A$. Suppose there is an inclusion 
$$
A_1 \lra I
$$
of $A$-modules 
inducing an equivalence 
$$
\gr(A_*) \simeq \gr(F_I^*(A)) 
$$
of graded objects, where $I/I^2$ is of pure weight $1$. Then $A_* = F_{I}^*A$; namely,  the filtration in question is the $I$-adic filtration.  
\end{thm}

For example, if $A$ is an augmented algebra, complete with respect to the augmentation ideal $I$, there can only be one (multiplicative) filtration on $A$ satisfying the conditions of Theorem~\ref{introadicunicity}, the $I$-adic filtration.

We will observe that the comultipliciation on the coordinate algebra of a formal group preserves this filtration, so that the formal group structure lifts uniquely as well. 

\subsection{Deformation to the normal cone}

Our next order of business is to study a deformation to the normal cone construction in the setting of derived algebraic geometry. In essence this takes a closed immersion  $\EuScript{X} \to \EuScript{Y}$ of classical schemes and gives a $\GG_m$-equivariant family of formal schemes over  $\AAA^1$, generically equivalent to the formal completion $\widehat{\EuScript{Y}_{\EuScript{X}}}$ which degenerate to the normal bundle of $N_{\EuScript{X}|\EuScript{Y}}$ formally completed at the identity section. When applied to a formal group, $\formalgroup$ produces a $\GG_m$-equivariant $1$-parameter family of formal groups over the affine line.  

\begin{rem}
A construction of deformation to the normal cone of a similar nature has already appeared in the book of Rozenblyum and Gaitsgory, \textit{cf}.\ \cite{gaitsgory2017study}, in  characteristic zero. Here we make no such restrictions on characteristic, and therefore the following result does not follow directly from their work. 
\end{rem}

\begin{thm}\label{maintheorem1}
Let $f\colon \Spec(R) \to \formalgroup$ be the unit section of a formal group $\formalgroup$. Then there exists a  stack $\Def_{\filstack}(\formalgroup) \to \AAA^1/ \GG_m$ such that there is a map
$$
 \AAA^1 / \GG_m \lra \Def_{\filstack}\left(\formalgroup\right)
$$
whose fiber over $1 \in \AAA^1 / \GG_m$ is 
$$
\Spec(k) \lra \formalgroup
$$
and whose fiber over $0 \in \AAA^1/\GG_m$ is 
$$
\Spec(R)  \lra \widehat{T_{\formalgroup| R}} \simeq \widehat{\GG_a},
$$
the formal completion of the tangent Lie algebra of\, $\formalgroup$. 
\end{thm}

We would like to point out that the constructions occur in the derived setting, but the outcome is a degeneration between formal groups, which belongs to the realm of classical geometry.  
One may then apply the aforementioned  \emph{filtered Cartier duality} to this construction to obtain a group scheme $\Def_{\filstack}(\formalgroup)^\vee$ over $\filstack$, thereby equipping the cohomology of the (classical) Cartier dual $\formalgroup^\vee$ with a canonical filtration. 

By \cite[Proposition 7.3]{geometryofilt}, $\OO(\Def_{\filstack}(\formalgroup))$ acquires the structure of a complete filtered algebra (completeness follows by Proposition~\ref{Prop:5.12}).  We have the following characterization of the resulting filtration on $\OO(\Def_{\filstack}(\formalgroup)$ relating the deformation to the normal cone construction with the $I$-adic filtration of Theorem~\ref{introadicunicity}. 

\begin{cor}\label{Adicfiltrationdefcone}
Let $\formalgroup$ be a formal group over $k$. Then there exists a unique filtered formal group with  $\OO(\formalgroup)$ as its underlying object. In particular, there is an equivalence 
$$
\OO(\Def_{\filstack}\left(\formalgroup\right) \simeq F^*_{I}A 
$$ 
of abelian cogroup objects in $\on{CAlg}(\widehat{\on{Fil}}_k)$. Here the right-hand side denotes the $I$-adic filtration on the coordinate ring $A$, for $I$ the augmentation ideal corresponding to the inclusion of the unit in $\formalgroup$. 
\end{cor}

Hence the deformation to the normal cone construction applied to a formal group $\formalgroup$ produces a \emph{filtered formal group}.

Next, we  specialize to the case of the formal multiplicative group $\widehat{\GG_m}$.  By putting Cartier duality together with Corollary~\ref{Adicfiltrationdefcone}, we recover the filtration on the group scheme 
$$
\mathsf{Fix} := \on{Ker}(F -1 \colon \W(-) \lra \W(-)) 
$$
of Frobenius fixed points on the Witt vector scheme. In particular, we show that this filtration arises, via Cartier duality, from a certain $\GG_m$-equivariant family of formal groups over $\AAA^1$. As a consequence, the formal group defined is precisely an instance of the  deformation to the normal cone of the unit section $\Spec(k)\to \widehat{\GG_m}$. 

\begin{thm}[\textit{cf.}~Corollary~\ref{Corollary 6.1}] \label{Thm:1.6}
Let $\mathbb{H} \to \filstack$ be the filtered group scheme of\, \cite{moulinos2019universal}. This arises as the Cartier dual $\Def_{\filstack}(\formalgroup_m)^\vee$ of the deformation to the normal cone of the unit section $\Spec(k) \to \GG_m$. Namely, there exists an equivalence of group schemes over $\filstack$
$$
\Def_{\filstack}(\GG_m)^\vee \lra \mathbb{H}.
$$
Together with Corollary~\ref{Adicfiltrationdefcone}, this implies that the HKR filtration on Hochschild homology  is functorially induced, via filtered Cartier duality, by the $I$-adic filtration on $\OO(\formalgroup_m) \simeq k[[t]]$. 
\end{thm} 

\begin{rem}
As a consequence of the uniqueness of Corollary~\ref{Adicfiltrationdefcone}, the filtration on $\OO(\widehat{\GG_m})$ coming from the deformation to the normal cone coincides with the filtration due to Sekiguchi--Suwa in \cite{sekiguchi2001note} given by the formal completion of the degeneration of $\GG_m$ to $\GG_a$. Together with Theorem~\ref{Thm:1.6}, this implies that the deformation to the normal cone is, via the filtered Cartier duality constructed in this paper, the key geometric source of the filtration on Hochschild homology.  
\end{rem}

\subsection{Filtration on $\formalgroup$-Hochschild homology}
One may  of course apply the deformation to the normal cone construction to an arbitrary formal group over any base commutative ring. As a consequence, one obtains a canonical filtration on the aforementioned $\formalgroup$-Hochschild homology.  

\begin{cor}[\textit{cf.}~Theorem\ref{Thm:7.3}]
Let $\formalgroup$ be an arbitrary formal group. The functor 
$$
\on{HH}^{\formalgroup}(-)\colon \on{sCAlg}_R \lra \Mod_R
$$
admits a refinement to the $\infty$-category of filtered $R$-modules
$$
\widetilde{\on{HH}^{\formalgroup}(-)}\colon \on{sCAlg}_R \lra \Mod_R^{\on{filt}}
$$
such that
$$
\on{HH}^{\formalgroup}(-) \simeq \on{} \colim_{(\Z, \leq)}\widetilde{\on{HH}^{\formalgroup}(-)}. 
$$
In other words, $\on{HH}^{\formalgroup}(A)$ admits an exhaustive filtration for any formal group $\formalgroup$ and simplicial commutative algebra $A$. 
\end{cor}

\begin{rem}
Let $\formalgroup$ be a $1$-dimensional formal group. Then in this case, the associated graded of the filtration on $\on{HH}^{\formalgroup}(A)$ will be exactly the derived global sections on $\Map(B \widehat{\GG_a},\Spec(A))$, which is none other than the de Rham algebra $ \Sym(\mathbb{L}_{A|k}[1])$.
Thus, we see that $\on{HH}^{\formalgroup}(-)$ agrees with ordinary Hochshchild homology at the level of associated gradeds of the respective HKR filtrations. Any differences are thus detected via extensions.
\end{rem}

\subsection{A family of group schemes over the sphere}
We now shift our attention over to the topological context.  In \cite{ellipticII},  Lurie defines a notion of formal groups intrinsic to the setting of spectral algebraic geometry. We explore a weak notion of Cartier duality in this setup, between formal groups over an $E_{\infty}$-ring and affine group schemes, interpreted as group-like commutative monoids in the category of spectral schemes. Leveraging this notion of Cartier duality, we demonstrate the existence of a family of spectral group schemes for each height $n$. Since Cartier duality is compatible with base change, one rather easily sees that these spectral schemes provide lifts of various affine group schemes.  

In the following statement, $R^{\un}_{\formalgroup}$
denotes the \emph{spectral deformation ring} of the formal group $\formalgroup$, studied in \cite{ellipticII}. This   corepresents the formal moduli problem (in the setting of spectral algebraic geometry)  sending a complete (Noetherian) $E_{\infty}$ ring $A$ to the space of deformations of $\formalgroup$ to $A$ and is a spectral enhancement of the classical deformation rings of Lubin and Tate.

\begin{thm}
Let $\formalgroup$ be a formal group  of height $n$ over  $\Spec(k)$, for $k$ a finite field. Let $D(\formalgroup) := \formalgroup^\vee$ be its Cartier dual affine group scheme. Then there exists a functorial lift $\phi\colon D(\formalgroup^{\un}) \to \Spec(R^{\un}_{\formalgroup})$ giving  the following Cartesian square of affine spectral schemes:
$$
\xymatrix{
&D(\formalgroup) \ar[d]_{\phi'} \ar[r]^{p'} &  D(\formalgroup^{\un}) \ar[d]^{\phi}\\
 & \Spec(\mathbb{F}_p)   \ar[r]^{p}& \Spec(R^{\un}_{\formalgroup})\rlap{.}
}
$$
Moreover, $D(\formalgroup^{\un})$ will be a group-like commutative monoid object in the $\infty$-category of spectral stacks $\sStk_{R^{\un}_{\formalgroup}}$ over $R^{\un}_{\formalgroup}$. 
\end{thm}

The spectral group scheme  $D(\formalgroup^{\un})$ of the theorem arises as the weak Cartier dual of the universal deformation of the formal group $\formalgroup$; this naturally lives over $R^{\un}_{\formalgroup}$. 

We remark that a key example to which the above theorem applies is the restriction to $\mathbb{F}_p$ of the subgroup scheme $\mathsf{Fix}$ of fixed points on the Witt vector scheme, in height $1$. 

\subsection{Liftings of $\formalgroup$-twisted Hochshild homology}
Finally, we study an $E_\infty$ (as opposed to simplicial commutative) variant of $\formalgroup$-Hochshild homology. For an $E_\infty$ $k$-algebra, this will be defined in an analogous manner to $\on{HH}^{\formalgroup}(A)$ (see Definition~\ref{E_inftyvariant}). We conjecture that for a simplicial commutative algebra $A$ with underlying $E_\infty$-algebra,  denoted by $\theta(A)$, this recovers the underlying $E_\infty$-algebra of the simplicial commutative algebra $\on{HH}^{\formalgroup}(A)$. In the case of the formal multiplicative group $ \widehat{\GG_m}$, we verify this to be true, so that one recovers Hochschild homology. 

These theories now admit lifts to the associated spectral deformation rings. 

\begin{thm}
Let $\formalgroup$ be a height $n$ formal group over a finite field $k$ of characteristic $p$, and let $R^{\un}_{\formalgroup}$ be the associated spectral deformation $E_\infty$-ring.   Then there exists a functor  
$$
\on{THH}^{\formalgroup}\colon \on{CAlg}_{R^{\un}_{\formalgroup}} \lra \on{CAlg}_{R^{\un}_{\formalgroup}} 
$$
defined as 
$$
\on{THH}^{\formalgroup}(A):= R\Gamma( \Map_{\sStk_{R^{\un}_{\formalgroup}}}(B D(\formalgroup^{\un}), \Spec(A)   ), \OO ). 
$$
This lifts the $E_{\infty}$-variant of\, $\formalgroup$-Hochshild homology in the sense that if $A$ is a $k$-algebra for which there exists a $R^{\un}_{\formalgroup}$-algebra lift $\widetilde{A}$ with 
$$
\widetilde{A} \otimes_{R^{\un}_{\formalgroup}} k \simeq A,  
$$
then there is a canonical equivalence, cf.~Theorem~\ref{representability},  
$$
\on{THH}^{\formalgroup}(\widetilde{A}) \otimes_{R^{\un}_{\formalgroup}} k \simeq  \on{HH}_{E_{\infty}}^{\formalgroup}(A).
$$
\end{thm}

\begin{rem}
When $\formalgroup = \widehat{\GG_m}$, we show in Theorem~\ref{Thm:9.9} that $\on{THH}^{\widehat{\GG_m}}$ recovers the usual $\on{THH}$. 
\end{rem}

We tie the various threads of this work together in the speculative final section where we discuss the question of lifting the filtration on $\on{HH}^{\formalgroup}(-)$, defined in Section~\ref{additionstostory} as a consequence of the degeneration of $\formalgroup$ to $\filstack$, to a filtration on the topological lift $\on{THH}^{\formalgroup}(-)$. To this end, we conclude with a negative result in the case $\formalgroup= \widehat{\GG_m}$ (\textit{cf}.\ Proposition~\ref{negativeresult}) about lifting the filtered formal group  $\Def_{\filstack}(\widehat{\GG_m})$ to the sphere spectrum. 

\subsection{Future work}
We work over a ring of integers $\OO_K$ in a local field extension $K \supset \mathbb{Q}_p$. In this setting, one obtains a formal group, known as the \emph{Lubin--Tate formal group}, which is  canonically associated to a choice of uniformizer $\pi \in \OO_K$. In future work, we investigate analogs of the construction of $\mathbb{H}$ in \cite{moulinos2019universal}, which will be related by Cartier duality to this Lubin--Tate formal group. By the results of this paper, these filtered group schemes will have a canonical degeneration arising from the deformation to the normal cone construction of the Cartier dual formal groups. 

In another vein, we expect the study of these spectral lifts $\on{THH}^{\formalgroup}(-)$ to be an interesting direction. For example, there is the question of filtrations, and to what extent they lift to  $\on{THH}^{\formalgroup}(-)$. One could try to base change this along the map to the orientation classifier 
$$
R^{\un}_{\formalgroup} \lra R^{\on{or}}_{\formalgroup};
$$
\textit{cf}.\ \cite{ellipticII}. Roughly, this is a complex orientable $E_\infty$-ring with the universal property  that it classifies oriented deformations of the spectral formal group $\formalgroup^{\un}$; these are oriented in that they coincide with the formal group corresponding to a complex orientation on the underlying $E_\infty$-algebra of coefficients. For example, one obtains $p$-complete $K$-theory in height $1$. It is conceivable questions about filtrations and the like would be more tractable over this ring. 

\subsubsection*{Outline}
We begin in Section~\ref{sec2} with a short overview of the perspective on formal groups which we adopt. In Section~\ref{sec3}, we describe some preliminaries from derived algebraic geometry. In Section~\ref{sec4}, we construct the deformation to the normal cone and apply it to the case of the unit section of a formal group. In Section~\ref{sec5}, we apply this construction to the formal multiplicative group $\widehat{\GG_m}$ and relate the resulting degeneration of formal groups to constructions in \cite{moulinos2019universal}. In Section~\ref{deformationofGm}, we study resulting filtrations on the associated $\formalgroup$-Hochshild homologies. We begin Section~\ref{additionstostory} with a brief overview of the ideas which we borrow from \cite{ellipticII} in the context of formal groups spectral algebraic geometry, and we describe a family of spectral group schemes that arise in this setting that correspond to height $n$ formal groups over characteristic $p$ finite fields. In Section~\ref{spectralll}, we study lifts $\on{THH}^{\formalgroup}(-)$ of $\formalgroup$-Hochschild homology to the sphere, with a key input the group schemes of the previous section. Finally, we end with a short speculative discussion in Section~\ref{sec9} about potential filtrations on $\on{THH}^{\formalgroup}(-)$. 

\subsubsection*{Conventions}
We often work over the $p$-local integers $\Z_{(p)}$, and so we typically use $k$ to denote a fixed commutative $\Z_{(p)}$-algebra. If we use the notation $R$ for a ring or ring spectrum, then we are not necessarily working $p$-locally. In another vein, we work freely in the setting of $\infty$-categories and higher algebra from \cite{luriehigher}. We would also like to point out that our use of the notation $\Spec(-)$ depends on the setting; in particular, when we work with spectral schemes, $\Spec(A)$ denotes the spectral scheme corresponding to the $E_\infty$-algebra $A$. We will always be working in the commutative setting, so we implicitly assume all relevant algebras, coalgebras, formal groups, \textit{etc}.\ are (co)commutative.  Finally, for a fixed commutative ring $R$, we use the notation $\on{CAlg}_R$ to denote the $\infty$-category of all $E_\infty$-algebras over $R$, and the notation $\on{CAlg}^{\heartsuit}_{R}$ to denote the category of discrete commutative $R$-algebras. 

\subsection*{Acknowledgments}
I would like to thank Marco Robalo and Bertrand To\"{e}n for their collaboration in \cite{moulinos2019universal}, which led to many of the ideas presented in this work. 
I would also like to thank Bertrand To\"{e}n for various helpful conversations and ideas which have made their way into this paper. 

\section{Basic notions from derived algebraic geometry}\label{sec2}
In this section, we review some of the relevant concepts that we will use from the setting of derived algebraic geometry. We recall that there are two variants, one whose affine objects are connective $E_{\infty}$-rings, and one whose affine objects are simplicial commutative rings.  We review parallel constructions from both simultaneously, as we will switch between both settings.  

Fix a base commutative ring $R$, and let $\EuScript{C} = \{ \on{CAlg}^{\cn}_R, \on{sCAlg}_R \}$ denote either of the $\infty$-category of connective $R$-algebras or the $\infty$-category of simplicial commutative algebras. Recall that the latter can be characterized as the completion via sifted colimits of the category of (discrete) free $R$-algebras. Over a commutative ring $R$, there exists a functor
$$
\theta\colon \on{sCAlg}_R \lra \on{CAlg}^{\cn} 
$$
which takes the underlying connective $E_\infty$-algebra of a simplicial commutative algebra. This preserves limits and colimits so is in fact monadic and comonadic. 

In any case, one  may define a derived stack via its functor of points, as an object of the $\infty$-category 
$\on{Fun}(\EuScript{C}, \mathcal{S})$ satisfying hyperdescent with respect to a suitable topology on $\EuScript{C}^{\op}$, \textit{e.g.}, the \'{e}tale topology. From here on, we distinguish the context we  are working in by letting $\dStk_{R}$ denote the $\infty$-category of derived stacks and letting $\sStk_R$ denote the $\infty$-category of ``spectral stacks''. 

In either case, one obtains an $\infty$-topos, which is Cartesian closed, so that it makes sense to talk about internal mapping objects:  given any two $X,Y \in \on{Fun}(\EuScript{C}, \mathcal{S})$, one forms the mapping stack $\Map_{\EuScript{C}}(X, Y)$.  In various cases of interest, if the source and/or target is suitably representable by  a derived scheme or a derived Artin stack, then this is the case for $\Map_{\EuScript{C}}(X, Y)$ as well. 

There is a certain type of base-change result that we will use; \textit{cf}. \cite[Proposition A.1.5]{halpern2014mapping} and \cite[Proposition 9.1.5.7]{lurie2016spectral}.

\begin{prop}\label{Prop:2.1}
Let $f\colon \EuScript{X} \to \Spec(R)$ be a geometric stack over $\Spec(R)$ $($here $R$ is discrete$)$.   Assume  that one of the two conditions holds:
\begin{itemize}
    \item  $\EuScript{X}$ is a derived scheme. 
    \item  The morphism $f$ is of finite cohomological dimension over $\Spec(R)$, so that  the global sections functor sends $\QCoh(\EuScript{X})_{\geq 0}$ to  $({\Mod_R})_{\geq -n}$ for some positive integer $n$.
 
\end{itemize}
Then,  for $g\colon \Spec(R') \to \Spec(R)$, the diagram of stable $\infty$-categories 
$$
\xymatrix{
&\Mod_R  \ar[d]^{g^*} \ar[r]^{f^*} &  \QCoh(\EuScript{X}) \ar[d]^{g'^*}\\
 & \Mod_{R'}  \ar[r]^{f'^*} & \QCoh(\EuScript{X}_{R'})
}
$$
is right adjointable, and so, the Beck--Chevalley natural transformation  of functors  $ g^* f_*  \simeq   f'_*g'^{*}\colon \QCoh(\EuScript{X}) \to \Mod_{R'}$ is an equivalence.   
\end{prop}

\subsection{Formal algebraic geometry and derived formal descent}
In this paper, we will often find ourselves in the setting of formal algebraic geometry and formal schemes. Hence we recall some basic notions in this setting. We end this subsection with a notion of formal descent which is intrinsic to the derived setting. This phenomenon will be exploited in Section~\ref{deformationsection}. 

An (underived) \emph{formal affine scheme} corresponds to the following piece of data.

\begin{defn}
We define an adic $R$-algebra to be an $R$-algebra $A$ together with an ideal $I \subset A$ defining a topology on $A$.  We let $\on{CAlg}_{R}^{\heartsuit, \ad}$ denote the category of adic $R$-algebras. 
\end{defn}

\begin{const}
  Let $A$ be an adic commutative ring having a finitely generated ideal of definition $I \subseteq  A$. Then there exists a tower $\cdots \to A_3 \to A_2 \to A_1$ with the properties that 
\begin{enumerate}
    \item each of the maps $A_{i+1} \to A_i$ is a surjection with nilpotent kernel; 
    \item the canonical map $\colim \Map_{\on{CAlg}_R^{\heartsuit}}(A_n, B) \to \Map_{\on{CAlg}_R^{\heartsuit}}(A,B)$ induces an equivalence of the left-hand side with the summand of  $\Map_{\on{CAlg}_R^{\heartsuit}}(A,B)$ consisting of maps $\phi\colon A \to B $ annihilating some power of the ideal $I$; 
    \item each of the rings $A_i$ is finitely projective when regarded as an $A$-module. 
\end{enumerate}
One now defines $\on{Spf}(A)$ to be the filtered colimit 
$$
\colim_{i} \Spec(A_i)
$$
in the category of locally ringed spaces. In fact, $\on{Spf}(A)$  may be obtained as the left Kan extension of the $\Spec(-)$ functor along the inclusion $\on{CAlg}_{R}^{\heartsuit} \to \on{CAlg}_{R}^{\heartsuit, \ad}$. 
\end{const}

\begin{defn} 
A formal scheme  over $R$ is a functor 
$$
X\colon \on{CAlg}_R^{\heartsuit} \lra \on{Set}
$$
which is Zariski locally of the above form.  A (commutative) formal group is an abelian group object in the category of formal schemes. By Remark~\ref{abelianobjects1cat}, this consists of the data of a formal scheme $\formalgroup$ which takes values in groups, 
which commutes with direct sums. 
\end{defn}

There is a rather surprising descent statement one can make in the setting of  derived algebraic geometry. For this we first recall the notion of formal completion. 

\begin{defn}
Let $f\colon X \to Y$ be a closed immersion of locally Noetherian schemes. We define the formal completion to be the stack $\widehat{Y}_{X}$ whose functor of points is given by 
$$
\widehat{Y}_{X}(R)= Y(R) \times_{Y(R_{\red})} X(R_{\red}),  
$$
where $R_{\red}$ denotes the reduced ring $(\pi_0 R)_{\red}$.  
\end{defn}

Although defined in this way as a stack, this is actually representable by an object in the category of formal schemes, commonly referred to as the formal completion of $Y$ along $X$. 

We form the nerve $N(f)_\bullet$ of the map $f\colon X \to Y$, which we recall is a simplicial object that in degree $n$  is the $(n+1)$-fold product
$$
N(f)_n= X\times_Y X \times \cdots \times_Y X. 
$$
The augmentation map of this simplicial object naturally factors through the formal completion (by the universal property the formal completion satisfies). We borrow the following key  proposition from \cite{toen2014derived}. 

\begin{thm}\label{formaldescent}
The augmentation morphism $N(f)_\bullet \to \widehat{Y}_X$ displays $\widehat{Y}_X$ as the colimit of the diagram $N(f)_\bullet$ in the category of derived schemes. This gives an equivalence 
$$
\Map_{\dStk}(\widehat{Y}_X, Z) \simeq \lim_{n \in \Delta} \Map_{\on{dSch}}(N(f)_n, Z)
$$
for any derived scheme.  
\end{thm}

\begin{rem}
At its core, this is a consequence of \cite[Theorem 4.4]{carlsson2008derived} on derived completions in stable homotopy, giving a model for the completion of an $A$-module spectrum along a map of ring spectra $f\colon A \to B $  to be the totalization of a certain cosimplicial diagram of spectra obtained via a certain co-Nerve construction. See Bhatt's work on completions and derive de Rham cohomology in  \cite{bhatt2012completions} for related results. 
\end{rem}

\begin{warn}\label{derivedschemesnotstacks}
We emphasize that this augmentation $N(f)_\bullet \to \widehat{Y}_X$ satisfies a universal property with respect to mapping to \emph{derived schemes}, as opposed to \emph{derived stacks}, as indicated by the equivalence in the statement of Theorem~\ref{formaldescent}. 
\end{warn}

\subsection{Tangent and normal bundles}
Let $X$ be a derived stack and $E \in \Perf(X)$ a perfect complex of Tor-amplitude concentrated in degrees $[0,n]$. Then  we have the following notion; \textit{cf}. \cite[Section 3]{toen2014derived}. 

\begin{defn}
We define the linear stack associated to $E$ to be the space-valued functor with source affine derived schemes over $X$ 
$$ 
\mathbb{V}(E)\colon \on{dAff}^{\op}_{/X} \lra \mathcal{S}
$$ 
defined by
$$
(u\colon \Spec(A) \lra X) \longmapsto \Map_{\Mod_A}(u^*(E), A ). 
$$
We note that this becomes a derived stack  over $X$ as it satisfies \'{e}tale descent. 
\end{defn}

\begin{ex}\label{kgaexample}
Let $\OO_{X}[n] \in \Perf(X)$ be a shift of the structure sheaf. Then $\mathbb{V}(\OO_{X}[n])$ is simply $K(\GG_{a,X}, -n)$. For a general perfect complex $E$, this $\mathbb{V}(E)$ may be obtained by taking various twisted forms and finite limits of these $K(\GG_{a,X}, -n)$. 
\end{ex}

\begin{defn}
Let $f\colon X \to Y$ be a map of derived stacks, which we assume to be \emph{quasi-smooth}. This means that it is locally of finite presentation and the relative cotangent complex $L_{X|Y}$ has Tor-amplitude $(- \infty, 1]$. We define the normal bundle stack to be 
$$
\mathbb{T}_{X|Y}[1] := \mathbb{V}(L_{X|Y}[-1]).
$$
This will be a derived stack over $X$; if $f$ is a closed immersion of classical schemes, then this will be representable by the ordinary normal bundle. Furthermore, this can be expressed as the classifying stack of the tangent bundle stack $\mathbb{T}_{X|Y}[1]$. 
\end{defn}

\begin{ex}
Let $i \colon \Spec(R) \to \formalgroup$ be the unit section of a formal group. This is an l.c.i. closed immersion; hence the cotangent complex is concentrated in (homological) degree~$1$. Via the equivalence 
$$
\mathbb{V}(E) \simeq \Spec(\Sym(E))
$$
for $E$ of finite  nonnegative (homological) Tor-amplitude (\textit{cf}. \cite[Section 3]{toen2014derived}), 
we see that the normal bundle $\mathbb{T}_{\Spec(R) | \formalgroup}[1]$ is a derived scheme and is in fact equivalent to 
$$
\mathbb{T}_{\formalgroup}:= \mathbb{V}(\on{Lie}(\formalgroup)) ,
$$ 
the linear stack associated to the Lie algebra of $\formalgroup$. This may be checked at the level of functors of points. If moreover we are in the $1$-dimensional case, and if there is an orientation in the sense that there is a local equivalence of $\on{Lie}(\formalgroup) \simeq R$, then 

$$
\mathbb{T}_{\formalgroup} \simeq K(\GG_a,0)  = \GG_a,
$$ 
the additive group over $\Spec(R)$, at least after taking  formal completion.\end{ex}

\section{Formal groups and Cartier duality} \label{ogdiscussion}\label{sec3}
In this section, we review some ideas pertaining to the theory of (commutative) formal groups which will be used throughout this paper. In particular, we carefully review the notion of Cartier duality as introduced by Cartier in \cite{cartier1962groupes} and also described in \cite[Section 37]{hazewinkel1978formal}.  
We remark that one of the key contributions of this paper is to introduce filtered analogs of these results. 

There are several perspectives one may adopt when studying formal groups. In general, one may think of them as abelian group objects in the category of formal schemes or representable formal moduli problems. 
In this paper, we will be focusing on the somewhat restricted setting of formal groups which arise from certain types of Hopf algebras. In this setting, one has a particularly well-behaved duality theory which we will exploit. Furthermore, it is this structure which has been generalized by Lurie in \cite{ellipticII} to the setting of spectral algebraic geometry.

\subsection{Abelian group objects}
We start off with the notions of abelian group and commutative monoid objects in an arbitrary $\infty$-category and review their distinction. 

\begin{notation}
  For each $n \geq 0$,  let $\langle n \rangle$ denote the pointed set $\{1,\ldots,n, *\}$. Now let  $\on{Fin}_*$ denote the category whose objects are the sets $\langle n \rangle$ and whose morphisms are pointed maps. Finally, for $1 \leq i \leq n$, let $\rho^i\colon \langle n \rangle \to \langle 1 \rangle $ denote the morphism such that $\rho^i(j)=1$ if $i=j$ and $\rho^i(j)= *$ otherwise. 
\end{notation}

\begin{defn}
Let $\mathcal{C}$ be an $\infty$-category which admits finite limits. A commutative monoid object is a functor $M\colon \on{Fin}_* \to \mathcal{C}$ with the property that for each $n$, the natural maps $M(\rho(\langle n \rangle) \to M(\rho \langle 1 \rangle) $ induce equivalences $M(\rho \langle n \rangle) \simeq M(\langle 1 \rangle)^{n}$ in $\mathcal{C}$.

In addition, a commutative monoid  $M$ is group-like if for every object $C\in \mathcal{C}$, the commutative monoid $\pi_0 \Map(C,M)$ is an abelian group. 
\end{defn}
 
We now define the somewhat contrasting notion of abelian group object. This will be part of the relevant structure on a formal group in the spectral setting. 

\begin{defn}
Let $\mathcal{C}$ be an $\infty$-category admitting finite limits. Then the $\infty$-category $\on{Ab}(\mathcal{C})$ of abelian objects of $\mathcal{C}$ is defined to be  
$$
\on{Fun}^\times(\on{Lat}^{\op}, \mathcal{C}),
$$
the category of product-preserving functors from the category $\on{Lat}$ of finite rank, free abelian groups into $\mathcal{C}$.
\end{defn}

\begin{rem} \label{abelianobjects1cat}
Let $\mathcal{C}$ be a small (discrete) category with finite limits.  Then an abelian group object $A$ is such that its representable presheaf $h_A$ takes values in abelian groups. Furthermore, in this setting, the two notions of abelian groups and  group-like commutative monoid objects coincide; \textit{cf}. \cite[Warning 1.3.10]{ellipticII}
\end{rem}

\subsection{Formal groups and Cartier duality over a field} \label{cartierdualityfield}
Before setting the stage for the various manifestations of Cartier duality to appear,
we say a few things about Hopf algebras, as they are central to this work. We begin with a brief discussion of what happens over a field $k$. 

\begin{defn}
For us, a (commmutative, cocommutative) Hopf algebra $H$ over $k$ is an abelian group object in the category of (discrete) coalgebras over $k$.
\end{defn}

Unpacking the definition, and using the fact the category of coalgebras is equipped with a Cartesian monoidal structure (it is the opposite category of a category of commutative algebra objects), we see that this is just another way of identifying bialgebra objects $H$ with an antipode map 
$
i\colon H \lra H
$;
this arises from the ``abelian group structure'' on the underlying coalgebra.

\begin{const}
Let $H$ be a Hopf algebra. Then one may define a functor
$$
\on{coSpec}(H)\colon \on{CAlg}_k^{\heartsuit} \lra \on{Ab} , \quad R \longmapsto \on{Gplike}(H \otimes_k R)= \{x | \Delta(x) = x \otimes x\},
$$
assigning to a commutative $k$-algebra $R$ the set of group-like elements of $R \otimes_k H$. 
The Hopf algebra structure on $H$  endows these sets with an abelian group structure, which is what makes the above an abelian group object-valued functor. In fact, this will be a formal scheme, and there will be an equivalence 
$$
\on{coSpec} (H) \simeq \on{Spf}\left(H^\vee\right), 
$$
where $H^{\vee}$, the linear dual of $H$, is an $R$-algebra, complete with respect to an $I$-adic topology induced by an ideal of definition $I \subset R$. Hence we arrive at our first interpretation of formal groups; these correspond precisely to Hopf algebras.
\end{const}

\begin{const} \label{indproduality}
Let us unpack the previous construction from an algebraic vantage point. Over a field $k$, there is an equivalence 
$$
\on{cCAlg}_k \simeq \on{Ind}\left(\on{cCAlg}^{\fd}_k\right), 
$$
where $\on{cCAlg}^{\fd}_k$ denotes the category of coalgebras whose underlying vector space is finite-dimensional. By standard duality, there is an equivalence
$$
 \on{Ind}\left(\on{cCAlg}^{\fd}_k\right) \simeq \on{Pro}\left({\on{CAlg}^{\fd}_k}\right), 
$$
where we remark that $\on{cCAlg}^{\fd}_k \simeq ({\on{CAlg}^{\fd}_k})^{\op}$. This may then be promoted to a duality between abelian group/cogroup objects
\begin{equation} \label{eq:3.1}
\mathsf{Hopf}_k := \on{Ab}\left(\on{cCAlg}_k\right) \simeq \on{coAb}\left(\on{Pro}\left({\on{CAlg}^{\fd}_k}\right)\right). 
\end{equation}
\end{const}

\begin{rem}
The interchange of display (\ref{eq:3.1}) is precisely the underlying idea of Cartier duality of formal groups and affine group schemes. 
Recall that Hopf algebras correspond contravariantly via the $\Spec(-)$ functor to affine group schemes. Hence one has 

$$
\on{AffGp}_{k}^{\op} \simeq \mathsf{Hopf}_k \simeq \on{FG}_k,
$$
where the left-hand side denotes the category of affine group schemes over $k$.  The functor on the right is given by the functor $\on{coSpec}(-)$ described above.
We remark that in this setting, the category of Hopf algebras over the field $k$ is actually abelian; hence the categories of formal groups and affine group schemes are themselves abelian.  
\end{rem}

\subsection{Formal groups and Cartier duality over a discrete commutative ring}
The key results in this section are Propositions~\ref{equivalenceofcategoriesdiscrete} and~\ref{fullyfaithfultobecited} and 
Construction~\ref{keyconstructionofthissection}, which together imply a duality theory between formal groups and affine group schemes (over a discrete commutative ring). This is what we refer to as \emph{Cartier duality}. In addition, by Proposition~\ref{Prop:3.19} this duality is 
by taking group maps to  $\GG_m$ and $\widehat{\GG_m}$, respectively. In Section~\ref{sec8.1}, we will describe a generalization of these ideas over a base $E_\infty$-ring~R. 

Over a general commutative ring $R$, the duality theory between formal groups and affine group schemes is not quite as simple to describe. In practice, one restricts to certain subcategories on both sides, which then fit under the Ind-Pro duality framework of Construction~\ref{indproduality}. This will be achieved by imposing a condition on the underlying coalgebra of the Hopf algebras at hand. 

\begin{rem}
We study coalgebras following the conventions of \cite[Section 1.1]{ellipticII}. In particular, if $C$ is a coalgebra over $R$, we always require that the underlying $R$-module of $C$ is flat.  This is done as in \cite{ellipticII} to ensure that $C$ remains a coalgebra in the setting of higher algebra. Furthermore, we implicitly assume that all coalgebras appearing in this text are (co)commutative. 
\end{rem}

To an arbitrary coalgebra, one may functorially associate a presheaf on the category of affine schemes given by the cospectrum functor 
$$
\on{coSpec}\colon \on{cCAlg}_R \lra \on{Fun}(\on{CAlg}_R, \on{Set}).
$$

\begin{defn} \label{Def:3.11}
Let $C$ be a coalgebra. We define  $\on{coSpec}(C)$ to be the functor
$$
\on{coSpec}(C)\colon \on{CAlg}_R \lra \on{Set}
$$
defined by $R \mapsto \on{Gplike}(C \otimes_k R)= \{x | \Delta(x) = x \otimes x\}$. 
\end{defn}

The $\on{coSpec}(-)$ functor is fully faithful when restricted to a certain class of coalgebras. We borrow the following definition from \cite{ellipticII}. See also \cite{strickland1999formal} for a related notion of \emph{coalgebra with good basis}.

\begin{defn}  \label{smoothcoalgoriginal}
Fix $R$ and let $C$ be a (cocommutative) coalgebra over $R$. We say $C$ is \emph{smooth} if its underlying $R$-module is flat and if it is isomorphic to the divided power coalgebra 
$$
\Gamma^*_R(M):= \bigoplus_{n \geq 0} \Gamma^n_R(M)
$$
for some projective, finitely generated $R$-module $M$. Here, $\Gamma^n_R(M)$ denotes the invariants for the action of the symmetric group $\Sigma_n$ on $M^{\otimes n}$. 
 \end{defn}

Given an arbitrary coalgebra $C$ over $R$, the linear dual $C^\vee =\Map(C, R)$ acquires a canonical $R$-algebra structure. In general, $C$ cannot be recovered from $C^\vee$. However, in the smooth case, the dual $C$ acquires the additional structure of a topology on $\pi_0$, giving it the structure of an adic $R$ algebra. This allows us to recover $C$, via the following proposition; \textit{cf}. \cite[Theorem 1.3.15]{ellipticII}. 

\begin{prop} \label{equivalenceofcategoriesdiscrete}
Let $C, D \in \on{cCAlg}^{\sm}_R$ be smooth coalgebras. Then $R$-linear duality induces a homotopy equi\-va\-lence
$$
\Map_{\on{cCAlg}_R}(C, D) \simeq \Map^{\on{cont}}_{\on{CAlg}_R}(C^\vee, D^\vee).
$$
\end{prop}

\begin{rem}
One can go further and characterize intrinsically all adic $R$-algebras that arise as duals of smooth coalgebras. These will be equivalent to $\widehat{\on{Sym}^*(M)}$, the completion along the augmentation ideal  $\on{Sym}^{\geq 1}(M)$ for some  projective $R$-module $M$ of finite type.  
\end{rem}

\begin{rem}
Fix  a smooth coalgebra $C$. There is always a canonical map of stacks $\on{coSpec}(C) \to \Spec(A)$, where $A= C^\vee$, but it is typically not an equivalence. The condition that $C$ is smooth guarantees precisely that there is an induced equivalence $\on{coSpec}(C) \to \on{Spf}(A) \subseteq \Spec(A)$, where $\on{Spf}(A)$ denotes the formal spectrum of the adic $R$-algebra $A$. In particular $\on{coSpec}(C)$ is a formal scheme in the sense of \cite[Chapter~8]{lurie2016spectral}. 
\end{rem}

\begin{prop}[Lurie] \label{fullyfaithfultobecited}
Let $R$ be an commutative ring. Then the construction $C \mapsto \on{cSpec}(C)$ induces a fully faithful embedding of $\infty$-categories
$$
\on{cCAlg}^{\sm}_R \lra \on{Fun}(\on{CAlg}^{\heartsuit}_R, \mathcal{S}). 
$$
Moreover, this commutes with finite products and base change. 
\end{prop} 

\begin{proof}
This essentially follows from the fact that a smooth coalgebra can be recovered from its adic algebra. 
\end{proof}

\begin{notation}
In the following, $\on{CAlg}^{\ad}_R$ denotes the category of (discrete) adic-$R$-algebras; these are discrete $R$-algebras $A$ with an ideal $I$ for which the topology on $A$ generated by $I$ is complete. 
\end{notation}

\begin{const} \label{keyconstructionofthissection}
As a consequence of the fact that the $\on{coSpec}(-)$ functor preserves finite products, this can be upgraded to a fully faithful embedding of abelian group objects in smooth coalgebras into formal groups 
$$
\on{Ab}(\on{cCAlg}) \lra \on{Ab}(f\on{Sch}).
$$
Unless  mentioned otherwise, we will focus on formal groups of this form. Hence we use the notation $\on{FG}_R$ to denote the image of the above embedding, and so the term \emph{Cartier duality} refers to the equivalence between this and (abelian group objects in) smooth coalgebras. We summarize the above discussion with the following statement. 
\end{const}

\begin{thm}[Cartier duality]
There exists an equivalence of categories between formal groups $\on{FG}_R$ and the category of affine group schemes whose underlying Hopf algebra of functions is smooth as a coalgebra.   
\end{thm}

We would like to interpret the above correspondence geometrically. 
Let $\on{AffGrp}^{b}_R$ be the subcategory of affine group schemes corresponding via the $\Spec(-)$ functor to the category $\mathsf{Hopf}_{R}^{\,\sm}$, which we use to denote the category of Hopf algebras whose underlying coalgebra is smooth.  Meanwhile, a cogroup object $\widehat{H}$ in the category of  adic algebras  corepresents a functor 
$$
F\colon \on{CAlg}_{R}^{\ad} \lra \Grp, \quad A \longmapsto \Hom_{\on{CAlg}_{R}^{\ad}}(\widehat{H}, A), 
$$
where the group structure arises from the cogroup structure on $H$. 
Essentially by definition, this is exactly the data of a formal group, so we may identify the category of formal groups with the category $\on{coAb}(\on{CAlg}_{R}^{\ad})$. 

We have identified the categories in question as those of affine group schemes and formal groups, respectively; one can further conclude that these dualities are representable by certain distinguished objects in these categories. 

\begin{prop}[\textit{cf.} \protect{\cite[Propositions 37.3.6 and 37.3.11]{hazewinkel1978formal}}] \label{Prop:3.19}
There exist natural bijections 
$$
\Hom_{\mathsf{Hopf}_{A}^{\,\sm}}(A[t, t^{-1}], C \otimes A) \cong  \Hom_{\on{CAlg}_{R}^{\ad}}(D(C), A), 
$$
$$
\Hom_{\on{CoAb}({\on{CAlg}_B^{\ad}})}\left(B[[T]], A \widehat{\otimes_R} B\right) \cong \Hom_{\on{CAlg}_R}\left(D^T(A), B\right).
$$
Here, for a coalgebra $C$,  $D(C)$ is the  linear dual, and for any topological algebra $A$, $D^T(A)= \Map_{cont}(A, R)$ is the  \emph{continuous dual}. 
\end{prop}

One can put this all together to see that there are duality functors which are moreover represented by the multiplicative group and the formal multiplicative group, respectively. 

One has the following expected base-change property. 

\begin{prop}\label{basechange}
Let  $\formalgroup$ be a formal group over $\Spec(R)$, and suppose there is a map $f\colon R \to S$ of commutative rings. Let $\formalgroup_{S}$ denote the formal group over $\Spec(S)$ obtained by base change.  Then there is a natural isomorphism 
$$
D^T(\formalgroup|_{S}) \simeq D^T(\formalgroup)_{S} 
$$
of affine group schemes over $\Spec(S)$. 
\end{prop}

\section{Filtered formal groups} \label{filteredformalgroupsection}\label{sec4}
We define here a notion of a filtered formal group, along with Cartier duality for these. We discuss here only (``underived'') formal groups over discrete commutative rings, but we conjecture that these notions generalize to the case where $R$ is a connective $E_\infty$ ring. 

\subsection{Filtrations and $\boldsymbol{\filstack}$}
We first recall a few preliminaries about filtered modules over $E_\infty$-rings. 

\begin{defn}
Let $R$ be an $E_\infty$-ring. We set 
$$
\on{Fil}_R := \on{Fun}(\Z^{\op}, \Mod_R),
$$
where $\Z$ is viewed as a category with morphisms given by the partial ordering $\geq$, and we refer to this as the $\infty$-category of filtered $R$-modules.
\end{defn}

\begin{defn}
Let $R$ be an $E_\infty$-ring. We set 
$$
\on{Gr}_R := \on{Fun}(\Z^{\on{ds}}, \Mod_R),
$$
where $\Z^{\on{ds}}$ is viewed as discrete space, and we refer to this as the $\infty$-category of graded $R$-modules.
\end{defn}

\begin{rem}
The $\infty$-category $\on{Fil}_R$ is symmetric monoidal with respect to the Day convolution product. 
\end{rem}

\begin{defn}
There exist functors 
$$
\on{Und}\colon \on{Fil}_R \lra \Mod_R, \quad \gr\colon \on{Fil}_R \lra \on{Gr}_R
$$
such that to a filtered $R$-module $M$, one associates its underlying object $\on{Und}(M)= \colim_{n \to -\infty} M_n$ and $\gr(M) = \oplus_n \on{cofib}(M_{n+1} \to M_{n})$, respectively. 
\end{defn}

\begin{ex}
Let $A$ be a commutative ring and $I \subset A$ an ideal of $A$. We define a filtration $F^*_I(A)$ with 
$$
F^n_I(A) =
\begin{cases}
A, & n \leq 0,  \\
I^n,& n \geq 1. 
\end{cases}
$$
This is the \emph{I-adic} filtration on $A$. 
\end{ex}

\begin{defn} \label{completenessfiltered}
There exists a notion of completeness in the setting of filtrations. We say a filtered $R$-module $M$ is complete if 
$$
\lim_n M_n \simeq 0. 
$$
Alternatively, $M$ is complete if $\lim M_{-\infty}/M_n \simeq M_{- \infty} = \on{Und}(M)$. We denote the $\infty$-category of filtered modules which are complete by $\widehat{\on{Fil}}_R$. This will be a localization of $\on{Fil}_R$ and will come equipped with a completed symmetric monoidal, so that the \emph{completion} functor
$$
\widehat{(-)}\colon \on{Fil}_R \lra \widehat{\on{Fil}}_R 
$$
is symmetric monoidal.
\end{defn}
 
\begin{const}
The category of filtered $R$-modules, as an $R$-linear stable $\infty$-category, can be equipped with several different $t$-structures. We will occasionally work with the \emph{neutral} t-structure on $\on{Fil}_R$, defined so that 
$F^*(M) \in (\on{Fil}_R)_{\geq 0}$ if $F^n(M) \in \on (\Mod_{k})_{\geq 0}$ for all $n \in \Z$. Similarly,
$F^*(M) \in (\on{Fil}_R)_{\leq 0}$ if  $F^n(M) \in \on (\Mod_{R})_{\leq 0}$ for all $n \in \Z$.

We remark that the standard $t$-structure on $\Mod_R$ is compatible with sequential colimits (\textit{cf}. \cite[Definition 1.2.2.12]{luriehigher}). This has the consequence that if $F^*(M) \in \on{Fil}_R^{\heartsuit}$, then 
$$
\colim_{n \to - \infty} F^n(M) = \on{Und}( F^*(M))  \in \Mod_k^\heartsuit. 
$$
We occasionally refer to filtered $R$-modules with are in the heart of this $t$-structure as discrete.
\end{const}

We now briefly recall the description of filtered objects in terms of quasi-coherent sheaves over the stack $\filstack$. This quotient stack may be defined as the quotient of $\AAA^1 = \Spec(R[t])$ by the canonical $\GG_m = \Spec(R[t, t^{-1}])$-action induced by the inclusion $\GG_m \hookrightarrow \AAA^1$ of monoid schemes. This comes equipped with two distinguished points 
$$
1\colon \Spec(R) \cong \GG_m / \GG_m \lra \filstack, 
\quad
0\colon B \GG_m= \Spec(R) / \GG_m \lra \filstack, 
$$
which we often refer to in this work as the generic and special/closed point, respectively. 
We remark that the quotient map $\pi\colon \AAA^1 \to \filstack$ is a smooth (and hence fppf) atlas for  $\filstack$, making $\filstack$ into an Artin stack.  

\begin{thm}[\textit{cf.} \protect{\cite[Theorem 1.1]{geometryofilt}}]
There exists a symmetric monoidal equivalence 
$$
\on{Fil}_R \lra \QCoh(\filstack). 
$$
Furthermore, under this equivalence, one may identify the underlying object and associated graded functors with pullbacks along $1$ and $0$, respectively. 
\end{thm}
 
\subsection{Filtered formal groups}
We adopt the approach to formal groups in \cite{ellipticII} described above, where they are in particular smooth coalgebras $C$ with  
$$
C = \bigoplus_{i \geq 0} \Gamma^{i}(M), 
$$
where $M$ is a (discrete) projective module of finite type. Here, $\Gamma^n$ for each $n$ denotes the $\supth{n}$ divided power functor which for a dualizable module $M$ can be alternatively defined as 
$$
\Gamma^n(M):= \on{Sym}^n(M^\vee)^\vee, 
$$
that is to say, as the dual of the symmetric powers functor. 

\begin{const}
By the results of \cite{brantner2019deformation, raksit2020hochschild}, these can be extended to the $\infty$-categories  $\Mod_R$, $\on{Gr}(\Mod_R)$, $\on{Fil}(\Mod_R)$ of $R$-modules, graded $R$-modules and filtered $R$-modules, respectively. These are referred to as the \emph{derived divided powers} 

In particular, the $\supth{n}$ (derived) divided power functors
$$
\Gamma_{\gr}^n\colon \on{Gr}_R \lra \on{Gr}_R, \quad \Gamma_{\fil}^n\colon  \on{Fil}_R \lra \on{Fil}_R
$$
make sense in the graded and filtered contexts as well.  
\end{const}

\begin{defn}  \label{smoothfilteredcoalg}
A smooth filtered coalgebra is a coalgebra of the form
$$
C = \bigoplus_{n \geq 0} \Gamma_{\fil}^n(M),
$$
for $M$ a filtered $R$-module whose underlying object is a discrete projective $R$-module of finite type with $\gr(M)$ concentrated in nonpositive weights.
Note that this acquires a canonical coalgebra structure, as in \cite[Construction 1.1.11]{ellipticII}. Indeed, if we apply $\Gamma^*$ to $M \oplus M$, we obtain compatible maps 
$$
\Gamma^{n' + n''}(M \oplus M) \lra \Gamma^{n'}(M) \otimes \Gamma^{n''}(M), 
$$
where this is to be interpreted in terms of the Day convolution product. As in the unfiltered case in \cite[Construction 1.1.11]{ellipticII}, these assemble to give equivalences
$$
\Gamma^*(M \oplus M) \simeq \Gamma^*(M) \otimes \Gamma^*(M). 
$$
Via the diagonal map $M \to M \oplus M$ (recall $\on{Fil}(\Mod_k)$ is stable), this gives rise to a map
$$
\Delta\colon \Gamma^*(M) \lra \Gamma^{*}(M \oplus M) \simeq \Gamma^*(M) \otimes \Gamma^*(M)
$$
which one can verify exhibits $\Gamma^*(M)$ as a coalgebra in the category of filtered $k$-modules. 
\end{defn}

\begin{prop}
Let $M$ be a dualizable filtered $R$-module. Then the formation of divided power coalgebras is compatible with the associated graded and underlying object functors. 
\end{prop}

\begin{proof}
Let $\on{Und}\colon \on{Fil}_R \to \Mod_R$ and $\gr\colon \on{Fil}_R \to \on{Gr}_R$ denote the underlying object and associated graded functors, respectively. Each of these functors commutes with colimits and is symmetric monoidal. Thus, we are reduced to showing that each of these functors commutes with the divided power functor $\Gamma_{\fil}^n(-)$. For this, we use the following description of the divided powers (for an arbitrary dualizable object $M$):
$$
\Gamma_{\fil}^n(M) = \Sym_{\fil}^n(M^\vee)^\vee, 
$$
which is valid by \cite[Proposition 3.39]{brantner2019deformation}. 
The statement now follows from the fact that $\on{Und}$ and $\gr$, being symmetric monoidal, commute with dualizable objects and that they commute with $\on{Sym}^n$, which follows from the discussion in \cite[Remark~4.2.25]{raksit2020hochschild}. 
\end{proof}

\begin{defn}
The category of smooth filtered coalgebras  $\on{cCAlg}(\on{Fil}_k)^{\sm}$ is the full subcategory of filtered coalgebras generated by objects of this form. Namely, 
$C \in \on{cCAlg}(\on{Fil}_R)^{\sm}$ if there exists a filtered module $M$ which is dualizable, discrete and zero in positive degrees for which 
$$
C \simeq \bigoplus_{n \geq 0} \Gamma^n_{\fil}(M). 
$$
\end{defn}

\begin{rem}
The filtered module $M$ in the above definition is of the form 
$$ 
\cdots \supset M_{-2} \supset M_{-1} \supset M_0 \supset 0 \cdots,  
$$
which is eventually constant. 
\end{rem}
We now define the notion of a filtered formal group. 

\begin{defn} \label{filteredformalgroupdefinition}
A filtered formal group is an abelian group object in the category of smooth coalgebras. That is to say, it is a product-preserving functor 
$$
F\colon \on{Lat}^{\op} \lra \on{cCAlg}(\on{Fil}_R)^{\sm}. 
$$
\end{defn}

\begin{const}
Let $M \in \on{Fil}_R$ be a filtered $R$-module. We denote the (weak) dual $\underline{\Map}_{\fil}(M, R)$ by $M^{\vee}$. Note that if $M$ has a commutative coalgebra structure, then  this acquires the structure of a commutative algebra.
\end{const}

\begin{ex}
Let  $C = \oplus \Gamma_{\fil}^n(M)$ be a smooth coalgebra. Then one has an equivalence 
$$
C^\vee \simeq \left(\bigoplus \Gamma^n(M)\right)^\vee \simeq \prod_n \on{Sym}^n\left(M^\vee\right).   
$$
This is a complete filtered algebra.
\end{ex}

\begin{prop} \label{compatibilityofdual}
Let $C$ be a filtered smooth coalgebra, and let $C^\vee$ denote its $($filtered\,$)$ dual. Then at the level of the underlying object, there is an equivalence 
$$
\on{Und} C^\vee \simeq  \prod \on{Sym}^*(N)
$$
for some projective module $N$ of finite type.
\end{prop}

\begin{proof}
We unpack what the weak dual functor does on the $\supth{n}$ filtering degree of a filtered $R$-module. If $M \in \on{Fil}_R$, then this may be described as 
$$
\left(M^\vee\right)_{n} = \underline{\Map}_{\on{Fil}}(M, R)_n \simeq \fib\left(M_\infty^\vee \lra M^\vee_{1-n}\right), 
$$
where $M^\vee_\infty$ is the dual of the underlying $R$-module. Now let $M= C$ be a smooth coalgebra, so that 
$$
C= \bigoplus \Gamma^k(N) 
$$
for $N$ as in Definition~\ref{smoothfilteredcoalg}. Then $C$ is concentrated in negative weights; hence $C_{1-n}$ vanishes as $n \to - \infty$, so  
$$
\on{Und}(C^\vee) \simeq \colim_n \fib\left(C^\vee_\infty \lra C^\vee_{1-n}\right) \simeq \fib \colim_{n} \left(C_\infty^\vee \lra C_{1-n}^\vee\right) = \colim_n C^\vee_\infty.
$$
In particular, since $C_{1-n}$ eventually vanishes, we obtain the colimit of the constant diagram associated to  $C^\vee_\infty$. Hence 
$$
\on{Und}\left(C^\vee\right) \simeq \on{Und}(C)^\vee \simeq \prod_{m \geq 0} \on{Sym}_R^m(N). 
$$
This shows in particular that weak duality of these smooth filtered coalgebras commutes with the underlying object functor. 
\end{proof}

\begin{rem}
Proposition~\ref{compatibilityofdual} justifies the definition of smooth filtered coalgebras which we propose (\textit{cf}.\ Definition~\ref{smoothfilteredcoalg}). In general, it is not clear that weak duality commutes with the underlying object functor (although this of course holds true on dualizable objects). 
\end{rem}

\subsection{Filtered Cartier duality}
The following statement summarizes the results of the  rest of this section. 

\begin{thm}[Filtered Cartier duality] \label{Cartier duality statement}
The weak duality functor induces an equivalence 
$$
(-)^\vee\colon \on{FFG}_{R} \simeq \on{coAb}(\mathcal{D})), 
$$
where $\mathcal{D} $ is a full $($discrete$)$ subcategory of the $\infty$-category $\on{CAlg}(\widehat{\on{Fil}_R})$ consisting of commutative algebras in filtered $R$-modules. 
\end{thm}

The first step to proving Theorem~\ref{Cartier duality statement} is the following key proposition which states that the (weak) duality functor is fully faithful when restricted to the underlying smooth coalgebra of a filtered formal group. 

\begin{prop} \label{keyproposition}
The assignment 
$\on{cCAlg}^{\sm}(\on{Fil}_R) \to  \on{CAlg}(\widehat{\on{Fil}_R})$ given by 
$$
C \longmapsto C^\vee = \Map(C, R) 
$$
is fully faithful. 
\end{prop}

\begin{proof}
Let $D$ and $C$ be two  arbitrary smooth coalgebras. We would like to display an equivalence of mapping spaces 
\begin{equation} \label{maptobeshownasequivalence}
\on{Map}_{\on{cCAlg}^{\sm}(\on{Fil}_R)}(D, C) \simeq \Map_{\on{CAlg}(\widehat{\on{Fil}_R})}\left(C^\vee, D^\vee\right). 
\end{equation}  

Each of $C$ and $D$ may be written as a colimit, internally to filtered objects, 
$$
C \simeq \colim C_k, \quad D \simeq \colim D_m, 
$$
where 
$$
C_k = \bigoplus_{0 \leq i \leq k} \Gamma^i(M),  \quad D_m = \bigoplus_{0 \leq i \leq m} \Gamma^i(N). 
$$ 
Hence the map (\ref{maptobeshownasequivalence}) may be rewritten as a limit of maps of the form 
\begin{equation} \label{eq:4.2}
\Map_{\on{cCAlg}^{\sm}(\on{Fil}_R)}(D_m, C) \lra  \Map_{\on{CAlg}(\widehat{\on{Fil}_R})}\left(C^\vee, D_m^\vee\right). 
\end{equation}
The left-hand side of this may now be rewritten as 
$$
\Map_{\on{cCAlg}^{\sm}(\on{Fil}_R)}(D_m, \colim_k  C_k).  
$$
Now, the object $D_m$ will be compact by inspection (in fact, its underlying object is just a compact projective $k$-module), so that the above mapping space is equivalent to  
$$
\colim_k \Map_{\on{cCAlg}^{\sm}(\on{Fil}_R)}(D_m,  C_k).  
$$
We would now like to make a similar type of identification on the right-hand side of the map (\ref{eq:4.2}). 
For this, note that as a complete filtered algebra, $C^\vee \simeq \lim_k C_k^{\vee}$. 
Note  that there is a canonical map
$$
\colim_k \Map(C_k^\vee, D_m) \lra \Map(\lim C_k^\vee, D_m). 
$$
By Lemma~\ref{Lem:4.21} this is an equivalence.
Each term $C_k^{\vee}$, as a filtered object, is zero in high enough positive filtration degrees. As limits in filtered objects are created object-wise, one sees that the essential image of the above map consists of morphisms 
$$
\lim_k C_k^\vee  \lra C_j^\vee \lra D_m
$$
which factor through some $C_j^\vee$. Since $D_m$ is itself of the same form, then every map factors through some $C_j^\vee$. Hence we obtain the desired decomposition on the right-hand side of (\ref{eq:4.2}). It follows that the morphism of mapping spaces  (\ref{maptobeshownasequivalence}) decomposes into maps 
$$
\Map_{}(D_m, C_k) \lra \Map\left(C_k^\vee, D_m^\vee\right).
$$
These are equivalences because $D_j$ and $C_k$ are dualizable for every $j,k$, and the duality functor $(-)^\vee$ gives rise to an anti-equivalence between commutative algebra and commutative coalgebra objects whose underlying objects are dualizable. Assembling this all, we conclude that (\ref{maptobeshownasequivalence}) is an equivalence. 
\end{proof}

\begin{lem} \label{Lem:4.21}
The canonical map  of spaces 
$$
\colim \Map_{\on{CAlg}(\widehat{\on{Fil}_R})}\left(C_k^\vee, D_m \right) \lra \Map_{\on{CAlg}(\widehat{\on{Fil}_R})}\left(\lim\nolimits_k C_k^\vee, D_m\right) 
$$ 
induced by the projection maps $\pi_k\colon \lim_k C_k \to C_k$ is an equivalence.
\end{lem}

\begin{proof}
Fix an index $k$. We claim that the following is a pullback square of spaces:
\begin{equation} \label{pullbacksquare}
\xymatrix{
& \Map_{\on{CAlg}(\widehat{\on{Fil}_R})}\left(C_k^\vee, D_m \right)  \ar[d]^{\pi_k^*} \ar[r]^-{\on{Und}} &  \Map_{\on{CAlg}}\left(C_k^\vee, D_m \right) \ar[d]^{\pi_k^*}\\
 & \Map_{\on{CAlg}(\widehat{\on{Fil}_R})}\left(\lim_k C_k^\vee, D_m \right) \ar[r]^-{\on{Und}} & \Map_{\on{CAlg}}\left(\lim_k C_k^\vee, D_m \right)\rlap{.}
}    
\end{equation}
First note that even though $\on{Und}(-)$ does not generally preserve limits, it will preserve these particular limits by Proposition~\ref{compatibilityofdual}. To prove the claim, we see that the pullback  
$$
\Map_{\on{CAlg}(\widehat{\on{Fil}_R})}\left(\lim\nolimits_k C_k^\vee, D_m \right) \times_{\Map_{\on{CAlg}}(\lim_k C_k^\vee, D_m )}\Map_{\on{CAlg}}\left(C_k^\vee, D_m \right)
$$
parametrizes, up to higher coherent homotopy, ordered pairs $(f, g)$ with 

$$
f\colon \lim C_k^\vee \lra D_m 
$$
a map of filtered algebras and 
$$
g_k\colon C_k^\vee \lra D_m
$$
a map at the level of underlying algebras, such that there is a factorization of the underlying map 
$$
\on{Und}(f) \simeq \pi_k^*(g_k) = g_k \circ \pi_k
$$
along the map $\pi_k\colon \lim_k C_k^\vee \to C_k^\vee$. Recall that $\pi_k$ is also the underlying map of a morphism of filtered objects; since the composition  $\on{Und}(f) = g_k \circ \pi_k$ respects the filtration, this means that $g_k$ itself must respect the filtration as well.  This in particular gives rise to an inverse  
$$
\Map_{\on{CAlg}(\widehat{\on{Fil}_R})}\left(\lim\nolimits_k C_k^\vee, D_m \right) \times_{\Map_{\on{CAlg}}(\lim_k C_k^\vee, D_m )}\Map_{\on{CAlg}}\left(C_k^\vee, D_m \right) \lra \Map_{\on{CAlg}(\widehat{\on{Fil}_R})}\left(C_k^\vee, D_m \right) 
$$
of  the canonical map  
$$
\Map_{\on{CAlg}(\widehat{\on{Fil}_R})}\left(C_k^\vee, D_m \right) \lra \Map_{\on{CAlg}(\widehat{\on{Fil}_R})}\left(\lim\nolimits_k C_k^\vee, D_m \right) \times_{\Map_{\on{CAlg}}(\lim_k C_k^\vee, D_m )}\Map_{\on{CAlg}}\left(C_k^\vee, D_m \right)
$$
induced by the universal property of the pullback, which proves the claim. 
Now let $P_k$ denote the fiber of the left vertical map of~\eqref{pullbacksquare}. One sees that the fiber of the map
$$
\Map_{\on{CAlg}(\widehat{\on{Fil}_R})}\left(\lim\nolimits_k C_k^\vee, D_m \right) \times_{\Map_{\on{CAlg}}(\lim_k C_k^\vee, D_m )}\Map_{\on{CAlg}}\left(C_k^\vee, D_m \right)
$$
of the statement is $\colim P_k$. We would like to show that this is contractible. By the claim, this is equivalent to $\colim P^{\on{und}}_k $, where $P^{\on{und}}_k$ for each $k$ is the fiber of the right-hand vertical map of 
\eqref{pullbacksquare}. By \cite[Proof of Theorem 1.3.15]{ellipticII}, this is contractible. We will be done upon showing that the essential image the map in the statement is all of $\Map_{\on{CAlg}}(\lim_k C_k^\vee, D)$. To this end, we note that the essential image consists of maps 
$$
\lim\nolimits_k C_k^\vee  \lra C_j^\vee \lra D_m
$$
which factor through some $C_j^\vee$. However, since the underlying algebra of $D_m$ is nilpotent, every map factors through such a $C_j^\vee$. 
\end{proof}

\begin{rem}
  We remark that this is ultimately an example of the standard duality between ind and pro objects of an $\infty$-category $\mathcal{C}$. Indeed, one has a duality between algebras and coalgebras in $\on{Fil}_k$ whose underlying objects are dualizable. The equivalence of Proposition~\ref{keyproposition} is an equivalence between certain full subcategories of $\on{Ind}(\on{cCAlg}^{ \omega,\fil})$ and $\on{Pro}(\on{CAlg}^{\omega, \fil})$.
\end{rem}

\begin{defn} \label{cogroupdefinition}
Let $\mathcal{D}$ denote the essential image of the duality functor of Proposition~\ref{keyproposition}. Then, we define the category of (commutative) cogroup objects $\on{coAb}(\mathcal{D})$ to just be an abelian group object of the opposite category (\textit{i.e.}, of the category of smooth filtered coalgebras). As $(-)^\vee$ is an anti-equivalence of $\infty$-categories, this implies that Cartesian products on $\on{cCAlg(\on{Fil}_k})^{\sm}$ are sent to coCartesian products on $\mathcal{D}$. Hence this functor sends group objects to cogroup object. We refer to an object $C \in \on{coAb}(\mathcal{D})$ as a \emph{filtered formal group}.
\end{defn}

\begin{rem}
If $C^\vee$ is discrete (which is the setting we are primarily concerned with for the moment), then a commutative cogroup structure on $C$ is none other than a (co)commutative comonoid structure on $C^\vee$, making it into a bialgebra in complete filtered $R$-modules. 
\end{rem}

We now complete the proof of Theorem~\ref{Cartier duality statement}.  

\begin{proof}[Proof of Theorem~\ref{Cartier duality statement}]
Let 
$$
(-)^\vee\colon \on{cCAlg}\left(\on{Fil}_R\right)^{\sm} \lra \mathcal{D}
$$
be the equivalence of Proposition~\ref{keyproposition}. As described in Definition~\ref{cogroupdefinition} above, this may now be promoted to an equivalence 
$$
(-)^{\vee}\colon \on{Ab}(\on{cCAlg}(\on{Fil}_R)^{\sm}) \lra \on{CoAb}(\mathcal{D}). 
$$
This gives the desired duality between $\on{FFG}_R$ and $\on{CoAb}(\mathcal{D})$. 
\end{proof}

\begin{rem}
We explain our usage of the term \emph{filtered Cartier duality}. As we saw in Section~\ref{cartierdualityfield}, classical Cartier duality gives rise to an (anti)-equivalence between formal groups and affine groups schemes, at least in the most well-behaved situation over a field.  An abelian group object in smooth filtered coalgebras will be none other than a filtered Hopf algebra. This is due to the fact that we ultimately still restrict to the a $1$-categorical setting where Remark~\ref{abelianobjects1cat} applies, so abelian group objects agree with group-like commutative monoids. Out of this, therefore, one may extract a relative affine group scheme over $\filstack$. Hence the equivalence of Theorem~\ref{Cartier duality statement} may be viewed as a correspondence between filtered formal groups and a full subcategory of relatively affine group schemes over $\filstack$.
\end{rem}

Next, we prove a unicity result on complete filtered algebra structures with underlying object a commutative ring $A$ and specified associated graded (\textit{cf}.\ Theorem~\ref{introadicunicity}). 

\begin{prop} \label{unicityalgebras}
Let $A$ be a discrete $R$-algebra which is complete with respect to the $I$-adic topology induced by some ideal $I \subset A$. Let $A_* \in \on{CAlg}(\widehat{\on{Fil}}_R)$ be a complete, exhaustive, multiplicative filtration of $A$ such that there is an identification $\iota\colon \gr^*(A_*) = \gr^*(F_I^*(A))$. Suppose the inclusion $A_1 \to A$ factors through the ideal $I \subset A$ so that there is an inclusion 
$$
A_1 \lra I
$$
of $A$-modules. Then this map can be promoted to a multiplicative morphism of filtrations $A_* \to F_I^*(A)$ inducing  $\iota$. Hence we may identify $A_*$ with the $I$-adic filtration.
\end{prop}

\begin{proof}
Let $A_*$ be a complete filtered algebra with these properties. Using the identification of $\gr^*(A_*)$, we inductively extend the map
$$
A_1 \lra I
$$
in the hypothesis to a map 
$$
A_* \lra F_I^*(A).
$$
In degree~$2$, for example, there is an induced map $A_2 \to I^2$, coming from the fact that the composition $A_2 \to A_1 \to I \to I/I^2$ vanishes, which in turn follows from the hypothesis on the associated graded of $A_*$. 

Thus, we obtain a map of complete filtered algebras $A_* \to F_I^*(A)$. By construction, this map of filtrations induces the map $\iota$ by passing to the associated graded. Since both filtered objects are complete and since the associated graded functor is conservative when restricted to complete objects, we deduce that the map 
$$
A_* \lra F_I^*(A)
$$
is an equivalence of filtered algebras.  
\end{proof}

\begin{rem}
In particular, we may choose $A_* \in \mathcal{D}$, the image of the duality functor from smooth filtered coalgebras. In this case, $I= \Sym^{\geq 1}(M)$, the augmentation ideal of  $\widehat{\Sym}(M)$ for $M$ some projective module of finite type. 
\end{rem}

Now let $\formalgroup$ be a formal group over $\Spec(R)$, and let $\OO(\formalgroup)$ be its complete adic algebra of functions. This acquires a comultiplication 
$$
\OO\left(\formalgroup\right) \lra \OO\left(\formalgroup\right) \,\widehat{\otimes}\, \OO\left(\formalgroup\right) 
$$
and counit
$$
\epsilon\colon \OO\left(\formalgroup\right) \lra R 
$$
making $\OO(\formalgroup)$ into a abelian cogroup object in $\mathcal{D}$. Let $I= \ker(\epsilon)$ denote the augmentation ideal. By Proposition~\ref{unicityalgebras}, there exists a unique filtration with $I$ in filtering degree~$1$ and associated graded $\gr^*(F_{I}(A))$, where  $F_{I}(A)$ is the $I$-adic filtration on $A$. This will be exactly the $I$-adic filtration $F_{I}(A)$ itself. 

We show that this filtered algebra inherits the cogroup structure as well. 

\begin{cor} \label{uniqueformalgroupstructure}
The comultiplication  
$$
\Delta \colon \OO\left(\formalgroup\right) \lra \OO\left(\formalgroup\right) \,\widehat{\otimes}\, \OO\left(\formalgroup\right)
$$ 
can be promoted to a map of filtered complete algebras. Thus, there is a unique filtered formal group---\textit{i.e.}, an abelian cogroup object in the category $\mathcal{D}$ with associated graded free on a module concentrated in weight~$1$ and with underlying object is $\OO(\formalgroup)$---whose underlying formal group is $\formalgroup$.
\end{cor}

\begin{proof}
We need to show that  the comultiplication
$$
\Delta\colon \OO\left(\formalgroup\right) \lra \OO\left(\formalgroup\right)\,\widehat{\otimes}\, \OO\left(\formalgroup\right) 
$$
preserves the adic filtration. Let us first assume  that the formal group is  $1$-dimensional and oriented so that $\OO(\GG) \simeq R[[x]]$. We remark that every formal group is locally oriented.  In this case, the formal group law is given in coordinates by the power series
$$
f(x_1, x_2) = x_1 + x_2 + \sum_{i, j \geq 1} a_{i,j} x^i y^j
$$
with suitable $a_{i,j}$. 
In particular, the image of the ideal commensurate with the filtration is contained in $I^{\otimes 2} = (x_1, x_2)$, the ideal commensurate with the filtration on $\OO(\formalgroup) \widehat{\otimes} \OO(\formalgroup) \cong R[[x_1, x_2]]$. Note that this is itself the $(x_1, x_2)$-adic filtration on $R[[x_1, x_2]]$. By multiplicativity, $\Delta(I^n) \subset I^{\otimes 2 n}$ for all $n$. This shows that $\Delta$ preserves the filtration,  
giving $F^*_{I}A$ a unique coalgebra structure compatible with the formal group structure on $\formalgroup$. The same argument works in higher dimensions.  
\end{proof}

\section{The deformation of a formal group} \label{deformationsection}\label{sec5}

\subsection{Deformation to the normal cone}
To a pointed formal moduli problem $\EuScript{X}$ (such as a formal group), one may associate an equivariant family over $\AAA^1$ whose fiber over $\lambda \neq 0$ recovers $\EuScript{X}$. We will use this construction further on  to produce filtrations on the associated Hochschild homology theories. The author would like to thank Bertrand To\"{e}n for the idea behind this construction, and in fact related constructions appear in \cite{toen2020classes}. A variant of this construction in the characteristic zero setting also appears in \cite[Chapter IV.5]{gaitsgory2017study}. We would also like to point out \cite{khan2018virtual}. 

The construction pertains to more than just formal groups. Indeed, let $\EuScript{X} \to \EuScript{Y}$ be a closed immersion of  locally Noetherian schemes. We construct a filtration on $\widehat{\EuScript{Y}_{\EuScript{X}}}$, the formal completion of $\EuScript{Y}$ along $\EuScript{X}$, with associated graded the shifted tangent complex $T_{\EuScript{X}|\EuScript{Y}}[1]$. 

The first, and key, ingredient underlying all this is a cogroupoid object $S^{0, \bullet}_{\fil}$. We remark that the $\Spec(-)$ functor here is to be taken in the sense of affine stacks; \textit{cf}.\ \cite{toen2006champs}.

\begin{const} \label{construction of filtered zero sphere}
Let $\iota\colon B \GG_m  \to \filstack$. Let 
$$
\phi\colon \OO_{\filstack} \lra 0_*\left(\OO_{B \GG_m}\right) 
$$ 
be the unit map of commutative algebra objects in $$\on{CAlg}(\QCoh(\filstack)) \simeq \on{CAlg}(\on{Fil}(\Mod_k)).
$$ Finally, let $N(\phi)_\bullet$ be the nerve of this map, viewed as a simplicial object in this $\infty$-category; by construction this will be a groupoid object in $\on{CAlg}(\on{Fil}(\Mod_k))$. We define 
\begin{equation} {\label{mapforcechnerve}}
S^{0, \bullet}_{\fil} := \Spec(N(\phi)_\bullet),
\end{equation}
which will be a cogroupoid object in the $\infty$-category of derived affine schemes over $\filstack$.  
\end{const}

\begin{rem}
We now give a more explicit description of this groupoid object in degree~$1$. In Construction~\ref{construction of filtered zero sphere} above, the structure sheaf $\OO_{\AAA^1 / \GG_m}$ may be identified with the graded polynomial algebra $k[t]$, where $t$ is of weight $1$. In degree~$1$, one obtains the  fiber product 
\begin{equation}\label{fiberprod}
    \OO_{\AAA^1 / \GG_m} \times_{0_{*}(\OO_{B \GG_m})} \OO_{\AAA^1 / \GG_m}, 
\end{equation}
which may be thought of as the graded algebra 
$$
k[t_1, t_2]/(t_1+t_2)(t_1-t_2)
$$
viewed as an algebra over $k[t]$. 
If we apply the $\Spec(-)$ functor relative to $\AAA^1 / \GG_m$, we obtain the scheme corresponding to the union of the diagonal and antidiagonal in the plane. The pullback of this fiber product to $\Mod_{k}$ is 
$$
k \times_{1^{*}0_{*}(\OO_{B \GG_m})} k \simeq k \times_{0} k = k \oplus k. 
$$
The pullback to $\QCoh(B \GG_m)$ is $k[\epsilon]/ \epsilon^2$, the trivial square-zero extension of $k$ by $k$. To see this, we pull back the fiber product  (\ref{fiberprod})  to $\QCoh(B \GG_m)$, which gives the homotopy Cartesian square  
$$
\xymatrix{
& k[\epsilon]/(\epsilon^2) \ar[d]_{} \ar[r]^{} &  k \ar[d]^{}\\
 & k  \ar[r]^{}& k \oplus k[1]
}
$$
in this category. Hence we may define 
$$
S^0_{\fil}:= \Spec_{\filstack}  \left(\OO_{\AAA^1 / \GG_m} \times_{0_{*}(\OO_{B \GG_m})} \OO_{\AAA^1 / \GG_m}\right)
$$
as the relative spectrum (over $\filstack$). 
\end{rem}
 
\begin{rem}
By construction, this admits a map
$$
S^0_{\fil} \lra \AAA^1/\GG_m 
$$
making it into a filtered stack, with generic fiber and special fiber described in the above proposition.
We remark that we may think of $S^{0}_{\fil}$ as the degree~$1$ part of a \emph{cogroupoid object} $S^{0, \bullet}_{\fil}$ in the $\infty$-category of (derived) schemes over $\filstack$; indeed, we may apply $\Spec(-)$ to the entire Cech nerve of the map~\eqref{mapforcechnerve}. We can then take mapping spaces out of this cogroupoid to obtain a groupoid object. 
\end{rem}

Now let $\EuScript{X} \to \EuScript{Y}$ be a closed immersion of locally Noetherian schemes, as above. We will focus our attention on the following derived mapping stack, defined in the category $\dStk_{\EuScript{Y} \times \filstack}$ of derived stacks over  $\EuScript{Y} \times \AAA^1/ \GG_m$:
$$
\Map_{\EuScript{Y} \times \AAA^1 / \GG_m }\left(S_{\fil}^0, \EuScript{X} \times \AAA^1 / \GG_m\right). 
$$
By composing with the projection map $\EuScript{Y}\times \AAA^1 / \GG_m \to \AAA^1 / \GG_m$, we obtain a map 
$$
\Map_{\EuScript{Y} \times \AAA^1 / \GG_m }\left(S_{\fil}^0, \EuScript{X}\right) \lra \AAA^1 / \GG_m
$$
allowing us to view this as a filtered stack. The next proposition identifies its fiber over $1 \in \AAA^1 / \GG_m$. 

\begin{prop} \label{henven1}
There is an equivalence 
$$ 
1^*\left( \Map\left(S_{\fil}^0, \EuScript{X}\right)\right) \simeq \EuScript{X}\times_{\EuScript{Y}} \EuScript{X}. 
$$
\end{prop}

\begin{proof}
By formal properties of base change of mapping objects of $\infty$-topoi, there is an equivalence 
$$
1^*\left( \Map\left(S_{\fil}^0, \EuScript{X}\right)\right)  \simeq \Map_{\EuScript{Y}}\left(1^{*}S_{\fil}^0, 1^*\left(\EuScript{X} \times \AAA^1 / \GG_m \right)\right). 
$$
The right-hand side is the mapping object out of a disjoint sum of final objects and therefore is directly seen to be equivalent to $\EuScript{X}\times_{\EuScript{Y}} \EuScript{X}$. 
\end{proof}

Next we identify the fiber over the ``closed point'' $0\colon B \GG_m \to \AAA^1/ \GG_m$. 

\begin{prop} \label{henven2}
There is an equivalence of stacks 
$$
0^*\left( \Map\left(S_{\fil}^0, \EuScript{X}\right)\right) \simeq T_{\EuScript{X}|\EuScript{Y}},
$$
where $T_{\EuScript{X}|\EuScript{Y}}$ denotes the relative tangent bundle of\, $\EuScript{X} \to \EuScript{Y}$.  
\end{prop}

\begin{proof}
We base change along the map 
$$
\Spec(k) \lra B \GG_m \lra \AAA^1 / \GG_m.
$$ 
Invoking again the standard properties of base change of mapping objects, we obtain the equivalence 
$$
0^*\left( \Map\left(S_{\fil}^0, \EuScript{X}\right)\right)  \simeq \Map_{\EuScript{Y}}\left(0^{*}S_{\fil}^0, 0^*\left(\EuScript{X} \times \AAA^1 / \GG_m \right)\right).
$$
By construction, we may identify $0^{*}S_{\fil}^0$ with $\Spec(k[\epsilon]/ \epsilon^2)$. Of course, this means that the right-hand side of the above display is precisely the relative tangent complex $T_{\EuScript{X}| \EuScript{Y}}$.  
\end{proof}

To summarize, we have constructed a cogroupoid object in the category of schemes over $\AAA^1 / \GG_m$, whose piece  in cosimplicial degree~$1$ is $S^0_{\fil}$, and formed the derived mapping stack
$$
\Map_{\EuScript{Y} \times \AAA^1 / \GG_m }\left(S_{\fil}^0, \EuScript{X} \times \AAA^1 / \GG_m\right),
$$
which will in turn be the degree~$1$ piece of a  groupoid object in derived schemes over $\filstack$.

\begin{const}\label{delooping}
Let $\EuScript{M}_{\bullet}:= \Map_{\EuScript{Y} \times \AAA^1 / \GG_m }(S_{\fil}^{0,\bullet}, \EuScript{X} \times \AAA^1 / \GG_m)$. Note that we can interpret the degeneracy map
$$
\EuScript{X} \times \AAA^1 / \GG_m \lra \Map_{\EuScript{Y} \times \AAA^1 / \GG_m }\left(S_{\fil}^0, \EuScript{X} \times \AAA^1 / \GG_m\right)
$$
 as the ``inclusion of the constant maps.'' We reiterate that this is a groupoid object in the $\infty$-category of (relative) derived schemes over $\AAA^1 / \GG_m$.  We  let
 $$
 \Def_{\filstack}(\EuScript{X}/ \EuScript{Y}) := \colim_{\Delta}\EuScript{M}_\bullet
 $$
 denote the colimit of this groupoid object.  Note that the colimit is taken in the $\infty$-category of derived schemes over $\filstack$ (as opposed to all of derived stacks).   
\end{const}

\begin{rem} \label{groupoidisarelativescheme}
We emphasize the point that $\EuScript{M}_{\bullet}$ is a groupoid object in relative derived schemes over $\filstack$. To see this, note that via Propositions~\ref{henven1} and~\ref{henven2}, we have identified the degree~$1$ piece as a relative derived scheme. As this is a groupoid object, each term $\EuScript{M}_n$ may be written as an $n$-fold fiber product of relative derived schemes. Since the $\infty$-category is closed under fiber products, we now see this fact. 
\end{rem}

By construction, $\Def_{\filstack}(\EuScript{X}/ \EuScript{Y)}$ is a derived scheme over $\AAA^1 / \GG_m$. The following proposition identifies its ``generic fiber'' with the formal completion $\widehat{\EuScript{Y}_{\EuScript{X}}}$ of $\EuScript{X}$ in $\EuScript{Y}$.

\begin{prop} \label{genericfibercompletion}
There is an equivalence 
$$
1^* \Def_{\filstack}(\EuScript{X}/ \EuScript{Y}) \simeq \widehat{\EuScript{Y}_{\EuScript{X}}}. 
$$
\end{prop}

\begin{proof}
As pullback commutes with colimits, this amounts to identifying the delooping in the category of derived schemes over $\EuScript{Y}$. Note again that all objects are schemes and not  stacks, so that this statement makes sense. By the above identifications, delooping the above groupoid corresponds to taking the colimit of the nerve $N(f)$ of the map $f\colon \EuScript{X} \to \EuScript{Y}$, a closed immersion. Hence it amounts to proving that 
$$
\colim_{\Delta^{\op}} N(f) \simeq \widehat{\EuScript{Y}_{\EuScript{X}}}. 
$$
This is precisely the content of Theorem~\ref{formaldescent}. 
\end{proof}

\begin{rem} \label{filteredderivedschemeandnotstack}
As discussed in Warning~\ref{derivedschemesnotstacks}, the formal completion
$\widehat{Y}_X$ acquires the universal property of the colimit 
$1^* \Def_{\filstack}(\EuScript{X}/ \EuScript{Y}) $ only upon restricting to derived schemes; \textit{i.e.}, there will be an equivalence 
$$
\Map_{\dStk}\left(\widehat{Y}_X, Z\right) \simeq
\Map_{\dStk}\left(1^* \Def_{\filstack}\left(\EuScript{X}/ \EuScript{Y}\right), Z\right) \simeq \lim_{n \in \Delta} \Map_{\on{dSch}}\left(1^* \EuScript{M}_n, Z\right) 
$$
whenever $Z \in \on{dSch}$.
For us this will not pose a problem because we will ultimately only be forming mapping stacks valued in derived schemes. 
\end{rem}

A consequence of Proposition~\ref{genericfibercompletion} is that the resulting object  is pointed by $\EuScript{X}$ in the sense that there is a well-defined map $\EuScript{X} \to \widehat{\EuScript{Y}_{\EuScript{X}}}$, arising from the structure map in the associated colimit diagram. This map is none other than the ``inclusion'' of $\EuScript{X}$ into its formal thickening. 

Our next order of business is, somewhat predictably at this point, to identify the fiber over $ B \GG_m$ of $\Def_{\filstack}(\EuScript{X}/ \EuScript{Y})$ with the normal bundle of $\EuScript{X}$ in  $\EuScript{Y}$. 

\begin{prop} \label{specialfiberformalgroup}
There is an equivalence 
$$
0^{*} \Def_{\filstack}(\EuScript{X}/ \EuScript{Y}) \simeq \widehat{\mathbb{V}\left(T_{\EuScript{X}|\EuScript{Y}}[1]\right)} =: \widehat{N_{\EuScript{X}|\EuScript{Y}}}
$$
in the $\infty$-category of derived schemes over $B \GG_m$ of our stack $\Def_{\filstack}(\EuScript{X}/ \EuScript{Y})$.
\end{prop}

\begin{proof}
First, we remark that the right-hand side, being (the formal completion of) a linear stack over $\EuScript{X}$, acquires a  $\GG_m$-action. This can be seen as follows: First note that at the level of functors of points, $\GG_m(A)= \Map_{\Mod_{A}}(A, A)$. The action for each $\Spec(A) \to \EuScript{X}$ on $\mathbb{V}(T_{\EuScript{X}|\EuScript{Y}}[1])$  is thus given by composition: 
$$
\Map_{\Mod_{A}}(A, A) \times \Map_{\Mod_{A}}\left(u^*\left(T_{\EuScript{X}|\EuScript{Y}}[1]\right), A\right) \lra \Map_{\Mod_{A}}\left(u^*\left(T_{\EuScript{X}|\EuScript{Y}}[1]\right), A\right). 
$$

Now we proceed with the proof. As in the proof of the previous proposition, it amounts to understanding the pullback along $\Spec(k) \to B \GG_m \to \filstack$ of the groupoid object $\EuScript{M}_\bullet$. This is given by 
$$
\EuScript{X} \leftleftarrows T_{\EuScript{X}| \EuScript{Y}}\cdots, 
$$
where we abuse notation and identify $T_{\EuScript{X}| \EuScript{Y}}$ with $\mathbb{V}(T_{\EuScript{X}| \EuScript{Y}})$.
Note that $T_{\EuScript{X}| \EuScript{Y}} \simeq \Omega_{\EuScript{X}}(T_{\EuScript{X}|\EuScript{Y}}[1])$, and so we may identify the above colimit diagram with the simplicial nerve $N(f)$ of the unit section $\EuScript{X} \to T_{\EuScript{X}|\EuScript{Y}}[1] \simeq N_{\EuScript{X}|\EuScript{Y}}$. The result now follows from another application of Theorem~\ref{formaldescent}.
\end{proof}

The following statement summarizes the above discussion. 

\begin{thm} \label{yoyoma}
Let $f\colon \EuScript{X} \to \EuScript{Y}$ be a closed immersion of schemes. Then there exists a filtered  stack $\Def_{\filstack}(\EuScript{X}/ \EuScript{Y}) \to \AAA^1/ \GG_m$ $($making it into a relative scheme over $\AAA^1/\GG_m)$ with the property that there exists a map
$$
\EuScript{X} \times \AAA^1 / \GG_m \lra \Def_{\filstack}(\EuScript{X}/ \EuScript{Y})
$$
whose fiber over $1 \in \AAA^1 / \GG_m$ is 
$$
\EuScript{X} \lra \widehat{\EuScript{Y}_{\EuScript{X}}}
$$
and whose fiber over $0 \in \AAA^1/\GG_m$ is 
$$
\EuScript{X}  \lra \widehat{N_{\EuScript{X}| \EuScript{Y}}}, 
$$
the formal completion of the unit section of\, $\EuScript{X}$ in its normal bundle.
\end{thm}

\subsection{Deformation of a formal group to its normal cone} \label{Sec:5.2}
Fix a (classical) formal group $\widehat{\GG}$. We now apply the above construction to the unit section of the formal group, $\iota\colon \Spec(k) \to  \widehat{\GG}$. Note that $\formalgroup$ is already formally complete along $\iota$. We set
$$
\Def_{\filstack}\left(\formalgroup\right) := \Def_{\filstack}\left( \Spec(k) / \formalgroup\right). 
$$
This will be a relative scheme over $\filstack$.

\begin{prop} \label{Prop:5.12}
Let $\Spec(k) \to \formalgroup$ be the unit section of a formal group. Then, the stack $\Def_{\filstack}(\formalgroup)$ of Construction~\ref{delooping}  is a filtered formal group.
\end{prop}

\begin{proof}
We will show that there exists a filtered dualizable (and discrete) $R$-module  $M$ for which 
$$
\OO\left(\Def_{\filstack}\left(\formalgroup\right)\right) \simeq \Gamma^*_{\fil}(M)^\vee \simeq \widehat{\on{Sym}_{\fil}^*}\left(M^\vee\right).
$$
As was shown above, there is an equivalence  
$$
\Def_{\filstack}\left(\formalgroup\right)_1 \simeq \formalgroup, 
$$
where the left-hand side denotes the pullback along $\Spec(k) \to \filstack$; hence we conclude that the underlying object of $\OO(\Def_{\filstack}( \Spec(k) / \formalgroup))$ is  of the form $k[[t]] \simeq \widehat{\on{Sym}^*}(M)$ for  $M$ a free $k$-module of rank $n$. We now identify the associated graded of the filtered algebra corresponding to $\OO(\Def_{\filstack}(\formalgroup))$. For this, we use the equivalence 
$$
\Def_{\filstack}\left(\formalgroup\right)_0 \simeq \widehat{T_{\GG|k}}
$$
of stacks over $B \GG_m$. We note that the right-hand side may indeed be viewed as a stack over $B\GG_m$, arising from the weight $-1$ action of $\GG_m$  by homothety on the fibers. This is the $\GG_m$-action which will be compatible with the grading on the dual numbers $k[\epsilon]$ (which appears in Construction~\ref{construction of filtered zero sphere}) such that $\epsilon$ is of weight~$1$. 
In particular, since $\formalgroup$ is an $n$-dimensional formal group, it follows that the associated graded is none other than 
$$
\on{Sym}_{\gr}^*(M(1)), 
$$
the graded symmetric algebra on the graded $k$-module $M(1)$, which is $M$ concentrated in weight $1$.

Now let $M^f(1)$ be the filtered $k$-module 
$$
M^f(1)= 
\begin{cases}
M^f(1)_n = 0,& n>1, \\
M^f(1)_n = M,& n \leq 1. 
\end{cases}
$$
 We claim that there is a map 
 $$
 M^f(1) \lra \OO\left(\Def_{\filstack}\left(\formalgroup\right)\right);
 $$
this will follow from the fact that $M$ is projective, and so there will be a lift to $F^{1}(\OO(\Def_{\filstack}(\formalgroup)))$. Passing to filtered objects, this means that one has the desired map 
 $$
 M^f(1) \lra \OO\left(\Def_{\filstack}\left(\formalgroup\right)\right). 
 $$
This then induces a map
$$
\widehat{\on{Sym}_{\fil}}\left(M^f(1)\right) \lra \OO\left(\Def_{\filstack}\left(\formalgroup\right)\right)
$$
since $\OO(\Def_{\filstack}(\formalgroup))$ is a filtered commutative algebra and in fact complete as a filtered object. We now claim that this map is an equivalence; this follows  by completeness and from the fact that the induced map on associated gradeds is an equivalence.  

We would now like to identify the filtered object $\widehat{\on{Sym}_{\fil}}(M^f(1))$ with the $I$-adic filtration on $\widehat{\on{Sym}}(M)$. We remark that we now find ourselves in the setup of a filtered augmented monadic adjunction of \cite[Example 5.39, Proposition 5.40]{brantner2019deformation}; within this formalism, the functorially defined adic filtration on a free polynomial algebra will coincide with the filtered $\Sym$ construction  on $M^f(1)$. This equivalence will persist upon taking completions. 
Hence we conclude that the  filtration on $\OO(\Def_{\filstack}(\formalgroup))$ is none other than the adic filtration of  $\widehat{\on{Sym}(M)}$ with respect to the augmentation ideal. Finally, by Corollary~\ref{uniqueformalgroupstructure}, this acquires a canonical abelian cogroup structure which is a filtered enhancement of that of $\formalgroup$, making $\Def_{\filstack}(\formalgroup)$ into a filtered formal group. 
\end{proof}

Now we combine this construction with the $\filstack$-parametrized Cartier duality of Section~\ref{filteredformalgroupsection}.  

\begin{cor} \label{Cor:5.13}
Let $\formalgroup$ be a formal group over $\Spec(k)$, and let $\formalgroup^\vee$ denote its Cartier dual. Then the cohomology $R\Gamma(\formalgroup^{\vee}, \OO)$ acquires a canonical filtration. 
\end{cor}

\begin{proof}
By Proposition~\ref{Prop:5.12}, the coordinate algebra $\OO(\Def_{\filstack}(\formalgroup)$ corresponds via duality to an abelian group object  in smooth filtered coalgebras. As we are in the discrete setting, this is equivalent to the structure of a group-like commutative monoid in this category. In particular, this is a filtered Hopf algebra object, so it determines a group stack $\Def_{\filstack}(\formalgroup)^\vee$ over $\filstack$. 
\end{proof}

\section{The deformation to the normal cone of \texorpdfstring{$\boldsymbol{\widehat{\GG_m}}$}{Ghat\textunderscore m}} \label{deformationofGm}

By the above, given any formal group $\widehat{\GG}$, one may define a filtration on its Cartier dual $\formalgroup^\vee = \Map( \widehat{\GG}, \widehat{\GG_m})$ in the sense of \cite{geometryofilt}. In the case of the formal multiplicative group, this gives a filtration on its Cartier dual $(\GG_m)^\vee =\mathsf{Fix}$. In \cite{moulinos2019universal}, the authors defined a geometric filtration on this affine group scheme (defined over a $\Z_{(p)}$-algebra $R$)  given by a certain interpolation between the kernel and fixed points of the Frobenius on the Witt vector scheme. We would like to compare the filtration on $\Map(\widehat{\GG_m}, \widehat{\GG_m})$ with this construction.

\begin{cor} \label{Corollary 6.1}
The geometric filtration defined on $\mathsf{Fix}$ is Cartier dual to the $(x)$-adic filtration on 
$$
\OO(\widehat{\GG_m}) \simeq R[[x]].
$$
Furthermore, this filtration corresponds to the 
deformation to the normal cone construction  
$\Def_{\filstack}(\widehat{\GG_m})$  on $\formalgroup_m$ and coincides with the filtration of\, \cite{sekiguchi2001note}.
\end{cor}

\begin{proof}
Let 
$$
\mathcal{G}_t= \Spec(R[X,1/(1 + tX))]. 
$$
This is an affine group scheme, with multiplication given by 
$$
X \longmapsto X \otimes 1 + 1 \otimes X + t X \otimes X; 
$$
one sees by varying the parameter $t$ that this is naturally defined over $\AAA^1$. If $t$ is invertible, then this is equivalent to $\GG_m$; if $t=0$, this is just the formal additive group ${\GG_a}$.  If we take the formal completion of this at the  unit section, we obtain a formal group $\widehat{\mathcal{G}_{t}}$, with corresponding formal group law
\begin{equation}
F(X,Y)= X + Y + tXY,     
\end{equation}
which we may think of as a formal group over $\AAA^1$. In \cite{sekiguchi2001note}, the authors describe the Cartier dual of the resulting formal group, for every $t  \in R$, as the  group scheme 
$$
\ker \left(F - t^{p-1}  \id\colon \W_p \lra \W_p\right). 
$$
These assemble, by way of the natural $\GG_m$-action on the Witt vector scheme $\W$,  to give a filtered group scheme $\mathbb{H} \to \filstack$, \textit{cf}. \cite[Definition 2.3.7]{moulinos2019universal}, whose classifying stack is the filtered circle. The algebra of functions $\OO(\mathbb{H})$ acquires  a comultiplication; by results of \cite{geometryofilt}, we may think of this as a filtered Hopf algebra.

Let us identify this filtered Hopf algebra a bit further; by abuse of notation, we refer to it as $\OO(\mathbb{H})$. After passing to  underlying objects, it is the divided power coalgebra $\bigoplus \Gamma^n(R)$. The algebra structure on this comes from the multiplication on $\widehat{\GG_m}$, via Cartier duality. On the graded side, we have the coordinate algebra of $\mathsf{Ker}$, which by \cite[Lemma 3.2.6]{drinfeld2020prismatization} is none other than the free divided power algebra 

$$
R \langle x \rangle \cong R\left[x, \frac{x^2}{2!},\ldots\right]. 
$$
One gives this the grading where each $\frac{1}{n!}x^n$ is of pure weight $-n$. The underlying graded smooth coalgebra is 
$$
\bigoplus_{n}\Gamma_{\gr}(R(-1)). 
$$
We deduce by weight reasons that there is an equivalence of filtered coalgebras
$$
\OO(\mathbb{H}) \simeq \bigoplus_n \Gamma^n\left(R^{f}(-1)\right), 
$$
where $R^{f}(-1)$
is trivial in filtering degrees $n > 1 $ and equal to $R$ otherwise. 

The consequence of the analysis of the above paragraph is that the Hopf algebra structure on $\OO(\mathbb{H})$ corresponds to the data of an abelian group object in smooth filtered coalgebras; \textit{cf.} Section~\ref{filteredformalgroupsection}. In other words, this is a filtered formal group, which is uniquely determined by its underlying formal group by Corollary~\ref{Cor:5.13}. In this case, it will be uniquely determined by $\widehat{\GG_m}$.  

Now we relate this to the deformation to the normal cone construction applied to $\widehat{\GG_m}$, which also outputs a  filtered formal group. Indeed, by the reasoning of Proposition~\ref{Prop:5.12}, this filtered formal group is itself given by the adic filtration  on $R[[t]]$ together with the filtered coalgebra structure uniquely determined by the group structure on $\widehat{\GG_m}$. 
\end{proof}

\section{\texorpdfstring{$\boldsymbol{\formalgroup}$}{Ghat}-Hochschild homology} \label{additionstostory}
As an application to the above deformation to the normal cone constructions associated to a formal group, we further somewhat the following proposal of \cite{moulinos2019universal} described in the introduction.

\begin{const}
Let $R$ be a $\Z_{(p)}$-algebra. Let $\formalgroup$ be a formal group over $R$. Its Cartier dual $\formalgroup^\vee$ is an affine commutative group scheme. We let $B \formalgroup^\vee$ denote the classifying stack of the group scheme $\formalgroup^\vee$. Let $X = \Spec(A)$ be an affine derived scheme, corresponding to a simplicial commutative $R$-algebra $A$. One forms the derived mapping stack
$$
\Map_{\dStk_R}\left(B \formalgroup^\vee, X\right).
$$
\end{const}

If $\formalgroup= \widehat{\GG_m}$, then by the affinization techniques of \cite{toen2006champs, moulinos2019universal}, one recovers, at the level of global sections,   
$$
R\Gamma\left(\Map_{\dStk_{R}}\left(B \widehat{\GG_m}^\vee, X\right), \OO \right) \simeq \on{HH}(A / R),
$$
the Hochschild homology of $A$ as the global sections of this construction. Following this example one can make the following definition (\textit{cf}. \cite[Section 6.3]{moulinos2019universal}). 

\begin{defn}
Let $\formalgroup$ be a formal group over $R$. Let 
$$
\on{HH}^{\formalgroup}\colon \on{sCAlg}_R \lra \Mod_R
$$
be the functor defined by 
$$
\on{HH}^{\formalgroup}(A) := R\Gamma\left(\Map_{\dStk}\left(B \formalgroup^\vee, X\right), \OO\right).
$$
\end{defn}

As was shown in  Section~\ref{Sec:5.2}, given a formal group $\formalgroup$ over a commutative ring $R$, one can apply a deformation to the normal cone construction to obtain a formal group $\Def_{\filstack}(\formalgroup)$  over $\AAA^1 / \GG_m$. By applying $\filstack$-parametrized Cartier duality, one obtains a group scheme over $\filstack$.

\begin{thm} \label{Thm:7.3}
Let $\formalgroup$ be an arbitrary formal group. The functor 
$$
\on{HH}^{\formalgroup}(-)\colon \on{sCAlg}_R \lra \Mod_R
$$
admits a refinement to the $\infty$-category of filtered $R$-modules
$$
\widetilde{\on{HH}^{\formalgroup}(-)}\colon \on{sCAlg}_R \lra \Mod_R^{\on{filt}}
$$
such that
$$
\on{HH}^{\formalgroup}(-) \simeq \colim_{(\Z, \leq)}\widetilde{\on{HH}^{\formalgroup}(-)}. 
$$
\end{thm}

\begin{rem}
We remark that for a $1$-dimensional $\formalgroup$, one recovers the de Rham algebra $\on{Sym}(\mathbb{L}_{A|R}[1])$ as the associated graded. Thus, the difference between $\on{HH}^{\formalgroup}$ in this case and ordinary Hochschild homology will be detected by extensions.  
\end{rem}

\begin{proof}[Proof of Theorem~\ref{Thm:7.3}]
Let $\Def_{\filstack}( \formalgroup)^\vee$ be the Cartier dual of the deformation to the normal cone $\Def_{\filstack}(\formalgroup)$. Form the mapping stack
$$
\Map_{\dStk_{/\filstack}}\left(B \Def_{\filstack}\left( \formalgroup\right)^\vee, X \times \filstack\right).
$$
This base changes along the map 
$$
1\colon \Spec(R) \lra \filstack
$$
to the mapping stack 
$$
\Map_{\dStk_R}\left(B \formalgroup^\vee, X\right),
$$
which gives the desired  geometric refinement. The stack $\Map_{\dStk_{/\filstack}}(B \Def_{\filstack}( \formalgroup)^\vee, X \times \filstack)$ is a derived scheme relative to the base $\filstack$. Indeed, it is nilcomplete and infinitesimally cohesive, and it admits an obstruction theory by the arguments of \cite[Section 2.2.6.3]{toen2008homotopical}. Finally, its truncation is the relative scheme $t_0 X \times \filstack$ over $\filstack$---this follows from the identification
$$
t_0 \Map\left( B \formalgroup^\vee, X\right) \simeq t_0 \Map\left(B \formalgroup^\vee, t_0 X\right)
$$
and from the fact that there are no nonconstant (nonderived)
maps $BG \to t_0 X$ for $G$ a group scheme. 

Hence  we conclude by the criteria of \cite[Theorem C.0.9]{toen2008homotopical} that this is a relative affine derived scheme. Since $\mathcal{L}_{\fil}^{\formalgroup}(X) \to \filstack$ is a relative affine derived scheme,  we conclude that 
$\mathcal{L}_{\fil}^{\formalgroup}(X) \to \filstack$ is of finite cohomological dimension, and so by Proposition~\ref{Prop:2.1}, $\widetilde{\on{HH}^{\formalgroup}(A)}$ defines an exhaustive filtration on $\on{HH}^{\formalgroup}(A)$. 
\end{proof}

\begin{rem}
In characteristic zero, all $1$-dimensional formal groups are equivalent to the additive formal group $\widehat{\GG_a}$, via an equivalence with its tangent Lie algebra. In particular, the above filtration splits canonically; one obtains an equivalence of derived schemes 
$$
\Map_{\dStk}\left(B \formalgroup^\vee, X\right) \simeq \mathbb{T}_{X|R}[-1]. 
$$
\end{rem}

In positive or mixed characteristic, this is of course not true. However, one can view all these theories as deformations along the map $B \GG_m \to \filstack $ of the de Rham algebra $\on{DR}(A)= \on{Sym}(\mathbb{L}_{A|R}[1])$. 

\section{Liftings to spectral deformation rings} \label{spectralll}

In this section, we lift the above discussion to the setting of spectral algebraic geometry over various ring spectra that parametrize \emph{deformations} of formal groups. These are defined in \cite{ellipticII} in the context of elliptic cohomology theory. As we will be switching gears now and working in this setting, we will spend some time recalling and slightly clarifying some of the ideas in \cite{ellipticII}. Namely, we introduce a correspondence between formal groups over $E_{\infty}$-rings and spectral affine group schemes, and we show it to be compatible with Cartier duality in the classical setting. We stress that the necessary ingredients already appear in \cite{ellipticII}. 

\subsection{Formal groups over the sphere}\label{sec8.1} 
We recall various aspects of the treatment of formal groups in the setting of spectra and spectral algebraic geometry. The definition is based on the notion of smooth coalgebra studied in Section~\ref{ogdiscussion}. 
In particular, the results of this section are generalizations to spectral algebraic geometry of the ideas of Sections 3.2 and 3.3. 

\begin{defn}
Fix an arbitrary $E_{\infty}$-ring $R$, and let $C$ be a coalgebra over $R$.   Recall that this means that $C \in \on{CAlg}(\Mod_R^{\op})^{\op}$. Then $C$ is smooth if it is flat as an $R$-module and if $\pi_0 C$ is  smooth as a  coalgebra over $\pi_0(R)$, as in Definition~\ref{smoothcoalgoriginal}.
\end{defn}

Given an arbitrary coalgebra $C$ over $R$, the linear dual $C^\vee =\Map(C, R)$ acquires a canonical $E_{\infty}$-algebra structure. In general, $C$ cannot be recovered from $C^\vee$. However, in the smooth case, the dual $C$ acquires the additional structure of a topology on $\pi_0$,  giving it the structure of an adic $E_{\infty}$-algebra. This allows us to recover $C$, via the following proposition; \textit{cf}. \cite[Theorem 1.3.15]{ellipticII}. 

\begin{prop} \label{spectralduality}
Let $C, D \in \on{cCAlg}^{\sm}_R$ be smooth coalgebras. Then $R$-linear duality induces a homotopy equivalence
$$
\Map_{\on{cCAlg}_R}(C, D) \simeq \Map^{\on{cont}}_{\on{CAlg}_R}\left(C^\vee, D^\vee\right).
$$
\end{prop}

\begin{rem}
One can go further and characterize intrinsically all adic $E_{\infty}$-algebras that arise as duals of smooth coalgebras. These (locally) have a formal power series ring as underlying  homotopy groups. 
\end{rem}

\begin{const}
Given a coalgebra $C \in \on{cCAlg}_R$, one may define a functor
$$
\on{cSpec}(C)\colon \on{CAlg}_R^{\cn} \lra \mathcal{S};
$$
this associates, to a connective $R$-algebra $A$, the space of group-like elements 
$$
\on{GLike}(A \otimes_{R}C) = \Map_{\on{cCAlg}_A}(A, A \otimes_R C).
$$
\end{const}

\begin{rem}
Fix a smooth coalgebra $C$. There is always a canonical map of stacks $\on{coSpec}(C) \to \Spec(A)$, where $A= C^\vee$, but it is typically not an equivalence. The condition that $C$ is smooth guarantees precisely that there is an induced equivalence $\on{coSpec}(C) \to \on{Spf}(A) \subseteq \Spec(A)$, where $\on{Spf}(A)$ denotes the formal spectrum of the adic $E_{\infty}$-algebra $A$. In particular, $\on{coSpec}(C)$ is a formal scheme in the sense of \cite[Chapter 8]{lurie2016spectral}.
\end{rem}

One has the following proposition, to be compared with Proposition~\ref{fullyfaithfultobecited}. 

\begin{prop}[Lurie]
Let $R$ be an $E_{\infty}$-ring. Then the construction $C \mapsto \on{cSpec}(C)$ induces a fully faithful embedding of $\infty$-categories
$$
\on{cCAlg}_{R}^{\sm} \lra \on{Fun}(\on{CAlg}^{\cn}_R, \mathcal{S}).
$$
\end{prop}

This facilitates the following definition of a formal group in the setting of spectral algebraic geometry. 

\begin{defn}
A functor $X\colon \on{CAlg}^{\cn}_R \to \mathcal{S}$ is a formal hyperplane if it is in the essential image of the $\on{coSpec}$ functor; we use the notation $\on{HypPlane}_R$ to denote the $\infty$-category of such objects. We now define a formal group to be an abelian group object in formal hyperplanes, namely an object of $\on{Ab}(\on{HypPlane}_R)$. 
\end{defn}

As is evident from the thread of the above construction, one may define a formal group to be a certain type  of Hopf algebra, but in a somewhat strict sense. Namely, we can define a formal group to be an object of $\on{Ab}(\on{cCAlg}^{\sm})$, that is, an abelian group object in the $\infty$-category of smooth coalgebras. We refer to these as \emph{strict Hopf algebras}. 

\begin{rem}
The monoidal structure on $\on{cCAlg}_R$ induced by the underlying smash product of $R$-modules is Cartesian; in particular, it is given by the product in this $\infty$-category. Hence a ``commutative monoid object'' in the category of $R$-coalgebras  will be a coalgebra that is additionally equipped with the structure of an $E_{\infty}$-algebra. In particular, it will be a bialgebra. 
\end{rem}

\begin{const} \label{Constr:8.9}
Let $\formalgroup$ be a formal group over an $E_{\infty}$-algebra $R$. Let $\H$  be a strict Hopf algebra $\H$ for which 
$$
\on{coSpec}(\H) = \formalgroup.
$$
Let 
$$U\colon \on{Ab}(\on{cCAlg}_R) \lra \on{CMon}(\on{cCAlg}_R)$$
be the forgetful functor from abelian group objects to commutative monoids. Since the monoidal structure on  $\on{cCAlg}_R$ is Cartesian, the structure of a commutative monoid in $\on{cCAlg}_R$ is that of a commutative algebra on the underlying $R$-module, and so we may view such an object as a bialgebra in $\Mod_{R}$. Finally, we apply $\Spec(-)$ (the spectral version) to this bialgebra to obtain a group object in the category of spectral schemes. This is what we refer to as the \emph{Cartier dual} $\formalgroup^{\vee}$ of $\formalgroup$.   
\end{const}

\begin{rem}
The above just makes precise, for a strict Hopf algebra $\H$ (\textit{i.e.}, an abelian group object), the association 
$$
\on{Spf}(H^\vee) \simeq \on{coSpec}(H) \longmapsto \Spec(H). 
$$
Unlike the $1$-categorical setting studied so far, there is no equivalence underlying this, as passing between abelian group objects to commutative monoid objects loses information; hence this is not a duality in the precise sense. In particular, it is not clear how to obtain a spectral formal group from a group-like commutative monoid in schemes, even if the underlying coalgebra is smooth. 
\end{rem}

\begin{prop} \label{basechangeinspectralsetting}
Let $R \to R'$ be a morphism of\, $E_{\infty}$-rings, and let $\formalgroup$ be a formal group over $\Spec(R)$ and $\formalgroup_{R'}$ its extension to $R'$. Then Cartier duality satisfies base change, so that there is an equivalence
$$
D\left(\formalgroup|_R'\right) \simeq D\left(\formalgroup\right)|_R.  
$$
\end{prop}

\begin{proof}
Let $\formalgroup = \on{Spf}(A)$ be a formal group corresponding to the adic $E_{\infty}$-ring $A$. Then the Cartier dual is given by $\Spec(\mathcal{H})$ for $\mathcal{H}= A^\vee$, the linear dual of $A$ which is a smooth coalgebra. The linear duality functor $(-)^\vee= \Map_R(-,R)$---for example by \cite[Remark 1.3.5]{ellipticII}---commutes with base change and is an equivalence between smooth coalgebras and their duals. Moreover, it preserves finite products and so can be upgraded to a functor between abelian group objects. 
\end{proof}

\subsection{Deformations of formal groups}
Let us recall the definition of a deformation of a formal group. These are all standard notions. 

\begin{defn}
Let $\widehat{\GG_0}$ be formal group defined over a finite field of characteristic $p$. Let $A$ be a complete Noetherian ring equipped with a  ring homomorphism $\rho\colon A \to k$, further inducing an isomorphism $A / \mathfrak{m} \cong k$. A deformation of $\widehat{\GG_0}$ along $\rho$ is a pair $(\widehat{\GG}, \alpha)$, where $\widehat{\GG}$ is a formal group over $A$ and $\alpha\colon \widehat{\GG_0} \to \widehat{\GG}|_k$ is an isomorphism of formal groups over $k$.  
\end{defn}

The data  $(\widehat{\GG}, \alpha)$ can be organized into a category $\on{Def}_{\widehat{\GG_0}}(A)$. 
The following classic theorem due to Lubin and Tate asserts that there exists a universal deformation, in the sense that there is a ring which corepresents the functor $A \mapsto \on{Def}_{\widehat{\GG_0}}(A)$. 

\begin{thm}[Lubin--Tate]
Let $k$ be a perfect field of characteristic $p$, and let $\widehat{\GG_0}$ be a $1$-dimensional formal group of height $n < \infty$ over $k$. Then there exist a complete local Noetherian ring $R^{\cl}_{\formalgroup}$, a ring homomorphism
$$
\rho\colon R^{\cl}_{\formalgroup} \lra k
$$
inducing an isomorphism $R^{\cl}_{\formalgroup}/ \mathfrak{m} \cong k$, and a deformation $(\widehat{\GG}, \alpha)$ along $\rho$ with the following universal property:
for any other complete local ring $A$ with an isomorphism $A \cong A / \mathfrak{m}$, extension of scalars induces  an equivalence 
$$
\on{Hom}_{k}\left(R^{\cl}_{\formalgroup}, A\right) \simeq \on{Def}_{\widehat{\GG_0}}(A, \rho)
$$
$($here, we regard the right-hand side as a category with only identity morphisms$)$. 
\end{thm}

For the purposes of this text, we can interpret the above as saying that every formal group over a complete local ring $A$ with residue field $k$ can be obtained from the universal formal group over $R^{\cl}_{\formalgroup}$ by base change along the map $R^{\cl}_{\formalgroup} \to A$. We let $\GG^{\un}$ denote the universal formal group over this ring.  

\begin{rem}
As a consequence of the classification of formal groups due to Lazard, one has a description
$$
A_0 \cong W(k)[[v_1,\ldots,v_{n-1}]],
$$
where the map $\rho\colon W(k)[[v_1,\ldots,v_{n-1}]] \to  k $ has kernel the  maximal ideal $\mathfrak{m}= (p, v_1,\ldots, v_{n-1})$. 
\end{rem}

\subsection{Deformations over the sphere } \label{Cartierdualoversphere}
As it turns out, the ring $A_0$ has the special property that it can be lifted to the $K(n)$-local sphere spectrum. To motivate the discussion, we  restate a classic theorem attributed to Goerss, Hopkins and Miller. We first set some notation.

\begin{defn}
Let  $\mathcal{FG}$ denote the category whose 
\begin{itemize}
    \item objects are pairs $(k, \widehat{\GG})$, where $k$ is a perfect field of characteristic $p$ and $\widehat{\GG}$ is a formal group over $k$, 
    \item morphisms from  $(k, \widehat{\GG})$ to  $(k', \widehat{\GG}')$ are  pairs $(f,\alpha)$, where $f\colon k \to k'$ is a ring homomorphism, and $\alpha\colon \widehat{\GG} \cong \widehat{\GG}'$ is an isomorphism of formal groups over $k'$. 
\end{itemize}
\end{defn}

\begin{thm}[Goerss--Hopkins--Miller]
There is a functor
$$
E\colon \mathcal{FG} \lra \on{CAlg}, \quad  \left(k, \widehat{\GG}\right) \longmapsto E_{k, \widehat{\GG}} 
$$
such that for every $(k, \widehat{\GG})$, the following hold: 
\begin{enumerate}
    \item $E_{k, \widehat{\GG}}$ is even periodic.
    \item The corresponding formal group over $\pi_0 E_{k, \widehat{\GG}}$ is the universal deformation of\, $(k, \widehat{G})$. In particular, $\pi_0 E_{k, \widehat{\GG}} \cong A_0 \cong \W(k)[[v_1,\ldots,v_{n-1}]] $. 
\end{enumerate}
\end{thm}

\noindent If we let $(l, \formalgroup) = (\F_{p^n}, \Gamma)$,
where $\Gamma$ is the $p$-typical formal group of height $n$, we set 
$$
E_n := E_{\F_{p^n}, \Gamma};
$$
this is the $\supth{n}$ \emph{Morava E-theory}. 

\begin{rem}
The original approach to this uses Goerss--Hopkins obstruction theory. A modern account due to Lurie can be found in \cite[Chapter 5]{ellipticII}.
\end{rem}

\begin{rem} 
As it turns out, this ring $R^{\cl}_{\formalgroup} \cong \pi_0 E_{k, \widehat{\GG}}$ can be thought of as parametrizing oriented deformations of the formal group $\widehat{\mathbb{G}}$. This oriented terminology, introduced in \cite{ellipticII}, roughly means that the formal group in question is equivalent to the Quillen formal group arising from the complex orientation on the base ring. However, there exists an $E_{\infty}$-algebra parametrizing \emph{unoriented deformations} of the formal group over $k$. 
\end{rem}

\begin{thm}[Lurie] \label{deformation}
Let $k$ be a perfect field of characteristic $p$, and let $\widehat{\GG}$ be a formal group of height $n$ over $k$. 
There exist a morphism of connective  $E_{\infty}$-rings
$$
\rho\colon R^{\un}_{\widehat{\GG}} \lra k
$$
and a deformation of\, $\widehat{\GG}$ along $\rho$ with the following properties: 
\begin{enumerate}
    \item $R^{\un}_{\widehat{\GG}}$ is Noetherian,  the induced map $\epsilon\colon \pi_0 R^{\un}_{\widehat{\GG}} \to k  $ is a surjection, and $R^{\un}_{\widehat{\GG}}$ is complete with respect to the ideal  $\ker(\epsilon)$. 
    \item Let $A$ be a Noetherian ring $E_{\infty}$-ring for which the underlying ring homorphism $\epsilon\colon \pi_0(A) \to k$ is surjective and $A$ is complete with respect to the ideal $\ker(\epsilon)$. Then extension of scalars induces an equivalence 
    $$
    \Map_{\on{CAlg}_/k}\left(R^{\un}_{\widehat{\GG}}, A\right) \simeq \on{Def}_{\widehat{\GG}}(A). 
    $$
\end{enumerate}
\end{thm}

\begin{rem}
We can interpret this theorem as saying that the ring $R^{\un}_{\formalgroup_0}$ corepresents the spectral formal moduli problem classifying deformations of $\formalgroup_0$. Of course, this then means that there exists a universal deformation (this is nonclassical) over $R^{\un}_{\formalgroup_0}$ which base changes to any other deformation of $\formalgroup$.
\end{rem}

\begin{rem}
This is actually proven in the setting of \emph{$p$-divisible groups} over more general algebras over $k$. However, the formal group in question is the identity component of a $p$-divisible group over $k$; moreover, any deformation of the formal group will arise as the identity component of a deformation of the corresponding $p$-divisible group (\textit{cf}. \cite[Example 3.0.5]{ellipticII}).
\end{rem}
 
Now fix an arbitrary formal group $\widehat{\GG}$ of height $n$ over a finite field $k$, and take its Cartier dual $D(\formalgroup) := \formalgroup^\vee$. From Construction~\ref{Constr:8.9}, we see that this is an affine group scheme over $\Spec(k)$. 

\begin{thm}
There exists a spectral scheme  $D(\formalgroup^{\un})$ defined over the $E_\infty$-ring $R^{\un}_{\widehat{\GG}}$, which lifts $D(\formalgroup)$, giving rise to the  following  Cartesian diagram of spectral schemes:

$$
\xymatrix{
&D\left(\formalgroup\right) \ar[d]_{\phi'} \ar[r]^{p'} &  D\left(\formalgroup^{\un}\right) \ar[d]^{\phi}\\
 & \Spec(k)   \ar[r]^{p}& \Spec\left(R^{\un}_{\formalgroup}\right)\rlap{.}
}
$$
\end{thm}

\begin{proof}
By Theorem~\ref{deformation} above, given a formal group $\formalgroup$ over a perfect field, the functor associating to an augmented ring $A \to k$ the groupoid of deformations  $\on{Def}(A)$ is corepresented by the spectral (unoriented) deformation ring  $R^{\un}_{\formalgroup}$. 
Hence we obtain a map 
$$
R^{\un}_{\formalgroup} \lra k
$$
of $E_{\infty}$-algebras over $k$. Over   $\Spec(R^{\un}_{\formalgroup})$, one has the universal deformation $\formalgroup_{\un}$. This base changes along the above map to $\formalgroup$. 
By definition, this formal group is of the form $\on{coSpec}(\mathcal{H})$ for some $\mathcal{H} \in \on{Ab}(\on{cCAlg}^{\sm}_{{R^{\un}_{\formalgroup}}})$. Let 
$$
U\colon \on{Ab}\left(\on{cCAlg}^{\sm}_{{R^{\un}_{\formalgroup}}}\right) \lra \on{CMon}^{\gp}\left(\on{cCAlg}^{\sm}_{{R^{\un}_{\formalgroup}}}\right) 
$$
be the forgetful functor from abelian group objects to group-like commutative monoid objects. We recall that the symmetric monoidal structure on cocommutative coalgebras is the Cartesian one. Hence group-like commutative monoids will have the structure of $E_{\infty}$-algebras in the symmetric monoidal $\infty$-category of $R^{\un}_{\formalgroup}$-modules. In particular, we obtain a commutative and cocommutative bialgebra, so we can take $\Spec(\mathcal{H})$; this will be a group-like commutative monoid object in the category of affine spectral schemes over $\Spec(R^{\un}_{\formalgroup})$. Since Cartier duality commutes with base change (\textit{cf}.~Proposition~\ref{basechangeinspectralsetting}), we conclude that $\Spec(\mathcal{H})$ base changes to $D(\formalgroup)$ under the map $ R^{\un}_{\formalgroup}$.  
\end{proof}

\begin{rem}
One might wonder about the possibility of lifting,  to the sphere spectrum, the filtration on $D(\formalgroup)$ given by the deformation to the normal cone. As we will see in Section~\ref{sec10}, this is substantially more subtle and fails  for the case $\formalgroup= \widehat{\GG_m}$. 
\end{rem}

\begin{ex}
  As a motivating example, consider $\formalgroup= \widehat{\GG_m}$, the formal multiplicative group over $\mathbb{F}_p$. As described in \emph{loc.\ cit.},
  this formal group is Cartier dual to $\on{Fix} \subset \W_p$, the Frobenius fixed point subgroup scheme of the Witt vectors $\W_p(-)$. This lifts to $R^{\un}_{\widehat{\GG_m}}$, which in this case is none other than the $p$-complete sphere spectrum $\mathbb{S}\hat{_p}$; \textit{cf}. \cite[Corollary 3.1.19]{ellipticII}. In fact, this object lifts to the sphere itself, by the discussion in \cite[Section 1.6]{ellipticII}. Hence we obtain an abelian group object in the category $\on{cCAlg}_{\mathbb{S}\hat{_p}}$ of smooth coalgebras over the $p$-complete sphere. Taking the image of this along the forgetful functor 
$$
\on{Ab}\left(\on{cCAlg}_{\mathbb{S}\hat{_p}}\right) \lra \on{CMon}\left(\on{cCAlg}_{\mathbb{S}\hat{_p}}\right), 
$$
we obtain a group-like commutative monoid $\mathcal{H}$ in $\on{cCAlg}_{\mathbb{S}\hat{_p}}$, namely a bialgebra in $p$-complete spectra. We set $\Spec(\mathcal{H}) = D(\formalgroup^{\un})$. Then 
base changing $\mathsf{Fix}^{\Sph}$ along the map
$$
\mathbb{S}\hat{_p} \lra \tau_{\leq 0}\mathbb{S}\hat{_p} \simeq \Z_p \lra \mathbb{F}_p
$$
recovers precisely the affine group scheme $D(\formalgroup^{\un})$, by the compatibility of Cartier duality with base change. 

One may even go further and base change to the orientation classifier (this is the $E_\infty$-ring classifying \emph{oriented} deformations of the formal group, which are compatible with a complex orientation; \textit{cf}. \cite[Chapter 6, Construction 6.0.1]{ellipticII})
$$
\mathbb{S}\hat{_p} \simeq R^{\un}_{\widehat{\GG_m}} \lra  R^{\on{or}}_{\widehat{\GG_m}} \simeq E_1 
$$
and recover height~$1$ Morava $E$-theory, a complex orientable spectrum. Moreover, in height~$1$, Morava $E$-theory is the $p$-complete complex $K$-theory spectrum $\KU\hat{_p}$. Applying the above procedure, one obtains the Hopf algebra corresponding to 
$$
C_{*}\left(\C P^{\infty}, \KU\hat{_p}\right), 
$$
whose algebra structure is induced by the abelian group structure on $\C P^{\infty}$. We now take the spectrum of this bialgebra; note that this is to be done in the  nonconnective sense (see \cite{lurie2016spectral}) as $\KU\hat{_p}$ is nonconnective. In any case, one obtains an affine nonconnective spectral group scheme 
$$
\Spec\left(C_{*}\left(\C P^{\infty}, \KU\hat{_p}\right)\right)
$$
which arises via the base change $\Spec(\KU_{\hat{p}}) \to \Spec(R^{\un}_{\formalgroup_m})$. 
We summarize this discussion with the following diagram of pullback squares in the $\infty$-category of nonconnective spectral schemes:
$$
\xymatrix{
&\on{Fix} \ar[d]_{\phi'} \ar[r]^{p'} &  \Spec(\mathcal{H}) \ar[d]^{\phi}  & \Spec\left( C_{*}\left(\C P^{\infty}, \KU\hat{_p}\right)\right) \ar[d] \ar[l] \\
 & Spec\left(\mathbb{F}_p\right)   \ar[r]^{p}& \Spec(R^{\un}_{\formalgroup})  & \Spec\left(\KU_{\hat{p}}\right)\rlap{.} \ar[l]
}
$$
Note that  we have the  factorization  
$$
\Sph\hat{_p} \lra ku\hat{_p} \lra \KU\hat{_p}
$$
through $p$-complete connective complex $K$-theory, so these lifts exists there as well. 
\end{ex}

\section{Lifts of \texorpdfstring{$\boldsymbol{\formalgroup}$}{Ghat}-Hochschild homology to the sphere} \label{sec9}
Let $ \formalgroup$ be a height $n$ formal group over a perfect field $k$. We study a variant of $\formalgroup$-Hochschild homology which is more adapted to the tools of spectral algebraic geometry. Roughly speaking, we take mapping stacks in the setting of spectral algebraic geometry over $k$, instead of derived algebraic geometry.

\begin{defn} \label{E_inftyvariant}
Let $\formalgroup$ be a formal group over $k$. We define the \emph{$E_{\infty}$-$\formalgroup$ Hochschild homology} to be the functor defined by 
$$
\on{HH}^{\formalgroup}_{E_\infty}(A)\colon \on{CAlg}_k^{\cn} \lra \on{CAlg}_k^{\cn}, \quad \on{HH}^{\formalgroup}_{E_\infty}(A)= R \Gamma\left(\Map_{\sStk_k}\left(B \formalgroup^{\vee}, \Spec (A)\right),\OO\right), 
$$
where $\Map_{\sStk_k}(-,-)$ denotes the internal mapping object of the $\infty$-topos $\sStk_k$. 
\end{defn}

It is not clear how the two notions of $\formalgroup$-Hochschild homology compare. 

\begin{conj}
Let $\formalgroup$ be a formal group and $A$ a simplicial commutative $k$-algebra. Then there exists a natural equivalence
$$
\theta(\on{HH}^{\formalgroup}(A)) \lra \on{HH}_{E_\infty}^{\formalgroup}(\theta(A))
$$
In other words, the underlying $E_\infty$-algebra of the $\GG$-Hochschild homology coincides with the $E_\infty$ $\formalgroup$-Hochschild homology of $A$, viewed as an $E_\infty$-algebra. 
\end{conj}

At least when $\formalgroup= \widehat{\GG_m}$, we know that this is true. In fact, this also recovers Hochschild homology (relative to the base ring $k$).

\begin{prop} \label{hochschildhomologygm}
Let $A$ be a simplicial commutative algebra over $k$. There is a natural equivalence  
$$
\theta(\on{HH}(A/k)) \simeq  \on{HH}^{\widehat{\GG_m}}_{E_\infty}(\theta(A))
$$
of\, $E_\infty$-algebra spectra over $k$.
\end{prop}

\begin{proof}
This is a modification of the argument of \cite{moulinos2019universal}. We have the (underived) stack $\mathsf{Fix} \simeq \widehat{\GG_m}^\vee$  and in particular a map 
$$
S^1 \lra B \mathsf{Fix} \simeq B\widehat{\GG_m}^\vee.
$$
This can also be interpreted, by Kan extension, as a map of spectral stacks. This further induces  a map between the mapping stacks
$$
\Map_{\sStk_k}(S^1, X) \lra \Map_{\sStk_k}\left(B\widehat{\GG_m}^\vee, X\right).
$$
We would like to show that this is an equivalence. In order to do this, we reduce to the case where $X = \AAA_{\sm}^1$, the (smooth) affine line.  
Recall that all connective $E_\infty$ $k$-algebras may be expressed as  colimits of free algebras, and all colimits of free algebras may be expressed as colimits of the free algebra on one generator $k\{t\}$. This follows from \cite[Corollary 7.1.4.17]{luriehigher}, where it is shown that $\on{Free}(k)$ is a compact projective generator for $\on{CAlg}_k$. These colimits become limits in the opposite category of derived affine schemes. As taking mapping stacks commutes with taking limits, we conclude that it is enough to test the above equivalence in the case where $X= \AAA_{\sm}^1$; this is the ``smooth'' affine line, \textit{i.e.}, $\AAA_{\sm}^1 = \Spec(k\{t\})$, the spectrum of the free $E_\infty$-$k$-algebra on one generator. For this we check that there is an equivalence on functors of points
$$
B \longmapsto \Map\left(B \widehat{\GG_m}^\vee \times B, \AAA^1\right)  \simeq  \Map\left(S^1 \times B, \AAA^1\right) 
$$
for each $B \in \on{CAlg}^{\cn_k}$. Each side may be computed as $\Omega^{\infty}( \pi_{*}\OO)$, where $\pi\colon B G \times B \to \Spec(k)$ denotes the structural morphism (where $G \in \{ \Z, \widehat{\GG_m}^\vee\}$).  The result now follows from the following two facts: 
\begin{itemize}
\item There is an equivalence of global sections $C^*(B \mathsf{Fix}, \OO) \simeq k^{S^1}$; \textit{cf}. \cite[Proposition 3.3.2]{moulinos2019universal}. 
\item  $B \mathsf{Fix}$ is of finite cohomological dimension; \textit{cf}. \cite[Proposition 3.3.7]{moulinos2019universal}.
\end{itemize}
We now obtain an equivalence on $B$-points
$$
\Omega^\infty\left(\pi_* \left(B\widehat{\GG_m}^\vee \times B\right)\right) \simeq   \Omega^\infty\left(\pi_*\left(B\widehat{\GG_m}^\vee\right) \otimes_k B \right)\simeq   \Omega^\infty\left(\pi_*\left(S^1\right) \otimes_k B \right)\simeq \Omega^\infty\left(\pi_* \left(S^1 \times B\right)\right).
$$
Note that the second equivalence follows from the finite cohomological dimension of $B\widehat{\GG_m}^\vee$. Applying global sections $R\Gamma(-, \OO)$ to this equivalence gives the desired equivalence of $E_\infty$-algebra spectra.  
\end{proof}

We now show that $\formalgroup$-Hochschild homology possesses additional structure which is already seen at the level of ordinary Hochshchild homology. Recall that for an $E_{\infty}$-ring $R$, its topological Hochschild homology may be expressed as the tensor with the circle: 
$$
\on{THH}(R) \simeq S^1 \otimes_{\Sph} R.
$$
Thus, when applying the $\Spec(-)$ functor  to the $\infty$-category of spectral schemes, this becomes a cotensor over~$S^1$. In fact, this coincides with the internal mapping object $\Map(S^1, X)$, where $X= \Spec(R)$. Furthermore, one has  the following base-change property of topological Hochshild homology: for a map $R \to S$ of $E_\infty$-rings, there is a natural equivalence
$$
\on{THH}(A/ R) \otimes_{R}S \simeq \on{THH}(A \otimes_R S/ S). 
$$
In particular, if $R$ is a commutative ring over $\mathbb{F}_p$ which admits a lift $\widetilde{R}$ over the sphere spectrum, then one has an equivalence 
$$
\on{THH}\left(\tilde{R}\right) \otimes_{\Sph} \mathbb{F}_p \simeq \on{HH}\left(R/ \mathbb{F}_p\right). 
$$
This can be interpreted geometrically as an equivalence of spectral schemes
$$
\Map\left(S^1, \Spec\left(\tilde{R}\right)\right) \times \Spec \left(\mathbb{F}_p\right) \simeq \Map\left(S^1, \Spec(R)\right)  
$$
over $\Spec (\mathbb{F}_p)$.
Let us show that such a geometric lifting occurs in many instances in the setting of $\formalgroup$-Hochschild homology.   

\begin{const}
Let $\formalgroup$ be a height $n$ formal group over $k$, and let $R$ be an commutative $k$-algebra. Let $\formalgroup_{\un}$ denote the universal deformation of $\formalgroup$, which is a formal group over   $R_{\formalgroup}^{\un}$. As in Section~\ref{Cartierdualoversphere}, we let $D(\formalgroup^{\un})$ denote its Cartier dual over this $E_\infty$-ring. 
\end{const}

\begin{thm}
Let $\formalgroup$ be a height $n$ formal group over $\mathbb{F}_p$, and let $X$ be an $\mathbb{F}_p$-scheme. Suppose there exists a lift $\tilde{X}$ over the spectral deformation ring $R^{\un}_{\formalgroup}$. Then there exists a homotopy pullback square of spectral algebraic stacks 
$$
\xymatrix{
&  \Map\left(B D\left(\formalgroup\right), X\right)    \ar[d]_{\phi'} \ar[r]^{p'} &  \Map\left(B D\left(\formalgroup^{\un}\right), \tilde{X}\right) \ar[d]^{\phi}\\
 & \Spec\left(\mathbb{F}_p\right)   \ar[r]^{p}& \Spec\left(R^{\un}_{\formalgroup}\right)
}
$$
displaying $\Map(B D(\formalgroup^{\un}), \tilde{X})$ as a lift of\, $\Map(B D(\formalgroup), X)$. 
\end{thm}

\begin{proof}
Given a map $p\colon X \to Y$ of spectral schemes, there is an induced morphism of $\infty$-topoi
$$
p^*\colon \Shv^{\et}_Y \lra \Shv^{\et}_X 
$$
Here the notation $ \Shv^{\et}_X$ denotes the big \'{e}tale site. This pullback functor is symmetric monoidal and moreover behaves well with respect to the internal mapping objects. Now let $X= \Spec (\mathbb{F}_p)$ and $Y= \Spec (R_{\formalgroup}^{\un})$, and let $p$ be the map induced by the universal property of the spectral deformation ring $R$.  In this particular case, this means  there will be an equivalence 
$$
p^* \Map\left(BD\left(\formalgroup^{\un}\right), \tilde{X}\right) \simeq  \Map\left(p^* B D\left(\formalgroup^{\un}\right), p^* \tilde{X}\right) \simeq \Map\left(B D\left(\formalgroup\right), X\right)
$$
since $\tilde{X} \times \Spec (\mathbb{F}_p) \simeq\ X $ and $p^* B {D(\formalgroup^{\un}) \simeq B D(\formalgroup)}$. 
\end{proof}

From this we conclude that the $\formalgroup$-Hochschild homology has a lift in the geometric sense, in that there is a spectral mapping stack over $\Spec (R^{\un}_{\formalgroup})$ which base changes to $\Map(B \formalgroup^{\vee}, X)$. We would like to conclude this at the level of global section $E_\infty$-algebras. This is not formal unless we have a more precise understanding of the regularity properties of $\Map(B D(\formalgroup^{\un}), X)$ for an affine spectral scheme $X= \Spec(A)$. 

Indeed, there is a map
\begin{equation} \label{eq:9.1}
 R\Gamma\left(\Map\left( B D\left(\formalgroup^{\un}\right),  \tilde{X}\right), \OO\right) \otimes \mathbb{F}_p \to
R\Gamma\left(\Map\left(p^* B D\left(\formalgroup^{\un}\right), p^* \tilde{X}\right), \OO\right), 
\end{equation}
but it is not \textit{a priori} clear that this is an equivalence. In particular, we have the  diagram of stable $\infty$-categories 
$$
\xymatrix{
& \Mod_{R^{\un}_{\formalgroup}} \ar[d]^{p^*} \ar[r]^-{\phi^*} &  \QCoh\left(\Map\left( B D\left(\formalgroup^{\un}\right),  \tilde{X}\right)\right) \ar[d]^{p'^*}\\
 & \Mod_{\mathbb{F}_p} \ar[r]^-{\phi'^*} & \QCoh\left(\Map\left( B D\left(\formalgroup^{\un}\right),  \tilde{X}\right)\right)
}
$$
for which we would like to verify that the Beck--Chevalley condition holds, \textit{i.e.}, that the canonically defined map
$$
\rho\colon p^* \circ  \phi_* \lra  \phi'_* \circ  p'^*
$$
is an equivalence. Here $\phi_*$ and  $\phi'_*$ are the right adjoints and may be thought of as global section functors. This construction applied to the structure sheaf $\OO$ recovers the map (\ref{eq:9.1}). 

This would follow from Proposition~\ref{Prop:2.1}   upon knowing either that the spectral stack $\Map(B D(\formalgroup^{\un}), \tilde{X})$ is representable by a derived scheme or, more generally, that it is of finite cohomological dimension. In fact,  it is the former. 

\begin{thm} \label{representability}
Let $\formalgroup$ be as above, and let $X$ denote an spectral scheme. Then the mapping stack
$\Map(B D(\formalgroup^{\un}), X)$
is representable by a spectral scheme. 
\end{thm}

\begin{proof}
This will be an application of the Artin--Lurie representability theorem; \textit{cf}. \cite[Theorem 18.1.0.1]{lurie2016spectral}. Given spectral stacks $X$ and $Y$, the derived spectral mapping stack
$
\Map(Y, X) 
$
is representable by a spectral scheme if and only if it is nilcomplete, infinitesimally cohesive and admits a cotangent complex and if the truncation $t_0(\Map(Y, X)  ) $ is representable by a classical scheme. By \cite[Proposition 5.10]{halpern2014mapping},  if $Y$ is of finite Tor-amplitude and  $X$ admits a cotangent complex, then so does the mapping stack $\Map(Y, X)$. In our case, $X$ is an honest spectral scheme which has a cotangent complex. Note that the condition of being of finite Tor-amplitude is local on the source with respect to  the flat topology (\textit{cf}. \cite[Proposition 6.1.2.1]{lurie2016spectral}. Thus if there exists a flat cover $U \to Y$ such that the composition $U \to Y \to \Spec(R)$ is of finite Tor-amplitude, then $Y \to \Spec(R)$ itself has this property. 
Infinitesimal cohesion follows from \cite[Lemma 2.2.6.13]{toen2008homotopical}. The following lemma takes care of the nilcompleteness. 

\begin{lem}
Let $Y$ be a spectral stack over $\Spec(R)$ which may be written as a colimit of affine spectral schemes  
$$
Y \simeq \colim \Spec(A_i), 
$$
where each $A_i$ is flat over $R$, and let $X$ be a nilcomplete spectral stack. Then $\Map_{\Stk_R}(Y,X)$ is nilcomplete. 
\end{lem}

\begin{proof}
The argument is similar to that of an analogous claim appearing in the proof of \cite[Theorem~2.2.6.11]{toen2008homotopical}. Let $Y$ be as above. Then
$$
\Map(Y, X) \simeq \lim_i \Map(\Spec(A_i), X), 
$$
and so it amounts to verify this  when $Y = \Spec(A_i) $ for $A_i$ flat. In this case, we see that for $B \in \on{CAlg}^{\cn}$,
$$
\Map(\Spec(A), X)(B) \simeq  X( A \otimes_{R} B).
$$
The map 
$$
\Map(\Spec(A), X)(B) \lra \lim \Map(\Spec(A), X)(\tau_{\leq n} B_n), 
$$
which we need to check is an equivalence, now translates to a map 
\begin{equation}\label{benchen}
X(A \otimes_R B) \lra X(\tau_{\leq n}B \otimes_R A).  
\end{equation}
We now use the flatness assumption on $A$. Using the general formula (\textit{cf}. \cite[Proposition 7.2.2.13]{luriehigher}) in this case
$$
\pi_n(A \otimes B) \simeq \on{Tor}^0_{\pi_0(R)}(\pi_0 A, \pi_n B), 
$$
we conclude that $\tau_{\leq n }(A \otimes B) \simeq A \otimes \tau_{\leq n} B$. Thus,~\eqref{benchen} above becomes a map
$$
X(A \otimes_R B) \lra X(\tau_{\leq n}(B \otimes_R A)),  
$$
which is an equivalence because $X$ was itself assumed to be nilcomplete. 
\end{proof}

Finally, we show that the truncation is an ordinary scheme. First of all, note that the truncation functor 
$$
t_0\colon \SStk \lra \Stk
$$
preserves limits and colimits. It is induced from the Eilenberg--Maclane functor 
$$
H\colon \on{CAlg}^0 \lra \on{CAlg}, \quad A \longmapsto HA, 
$$
which is itself adjoint to the truncation functor on $E_\infty$-rings. One sees that the truncation functor $t_0 = H^*\colon \SStk \to \Stk$ will have as a right adjoint the functor 
$$
\pi_0^*\colon \Stk \lra \SStk
$$
induced by the $\pi_0$ functor
$$
R \longmapsto \pi_0 R. 
$$
Thus it is right exact and preserves colimits. Hence if $Y = B G$ for some flat spectral group scheme $G$, then $t_0 BG \simeq B t_0 G$. Now, one has the identification  
$$
t_0 \Map(Y, X) \simeq \Map\left(t_0 Y, t_0 X\right) 
$$
in this particular case because $Y \simeq \colim \Spec(A_i)$ is a colimit of flat affine schemes, so this identification may be checked by looking at each component $\Map(\Spec(A_i), X)$ of the resulting limit. In this case, we can test this by hand or refer to \cite[Remark 5.1.3]{halpern2014mapping}. 

Thus we have the identification
$$
t_0  \Map\left(B D\left(\formalgroup^{\un}\right), X\right) \simeq  \Map(B G, t_0 X)
$$
for some (classical) affine  group scheme $G$. Recall that the only classical maps $f\colon B G \to t_0 X$ between a classifying stack and a scheme $t_0 X$ are the constant ones. Hence we conclude that the truncation of this spectral mapping stack is equivalent to the scheme $t_0 X$, the truncation of $X$.
\end{proof}

\subsection{Topological Hochschild homology}
As we saw, for a height $n$ formal group $\formalgroup$ over a finite field $k$, there exists a lift $D(\formalgroup^{\un})$ of the Cartier dual of $\formalgroup$; this allows one to define a lift of $\formalgroup$-Hochschild homology. Let us show that when the formal group is $\widehat{\GG_m}$, this lift is precisely topological Hochschild homology, at least after $p$-completion, as one would expect. For the remainder of this section, we let $\formalgroup= \widehat{\GG_m}$,  the formal multiplicative group. 

Let $X$ be a fixed spectral stack. We remark that there exists an adjunction of $\infty$-topoi 
$$
\pi^*\colon \mathcal{S} \rightleftarrows \SStk_{X} \colon\pi_*, 
$$
where on the right-hand side, one has the $\infty$-category of spectral stacks over $X$. 
In the following, we think of $S^1$ as a ``constant stack'' obtained by the adjunction
above.

\begin{prop}\label{Prop:9.8}
There exists a canonical map
$$
S^1 \lra B D\left(\formalgroup^{\un}\right)
$$
of group objects in the $\infty$-category of spectral stacks over $\Sph_{p}$. This gives a lift of the map 
$$
S^1 \lra B \Fix. 
$$
\end{prop}

\begin{proof}
By  \cite[Construction 3.3.1]{moulinos2019universal}, there is a canonical map
\begin{equation} \label{eq:9.3}
    \Z \lra \mathsf{Fix}
\end{equation}
in the category of fpqc abelian sheaves over $\Spec(\Z_{p})$. Note that this is in fact a map of ring objects, \textit{cf}. \cite[Appendix C.1.1]{drinfeldformalgroup}, and is unique as such, by the initiality of $\Z$. We would like to lift this to a map $\Z \to D(\formalgroup^{\un})$ of abelian group objects in $\SStk_{\Sph_{p}}$. We construct a lift directly and show that it base changes to the map $\Z \to \mathsf{Fix}$. For this we use the  construction of  Cartier duals of $\on{CMon}^{\gp}$-valued functors from \cite[Construction 3.7.1 and Proposition 3.9.6]{lurie2017elliptic}. Working in $\Stk_{\Sph_p}$, there is a canonical map of abelian group objects
$$
\formalgroup^{\un} \simeq \widehat{\GG_m} \xrightarrow{\;\;i\;\;} \GG_m
$$
given by the inclusion of the formal completion along the identity section, where $\GG_m = \Spec(\Sph[t, t^-1])$. Taking Cartier duals in the sense of \cite[Construction 3.7.1]{lurie2017elliptic}, we obtain a map
$$
D(\GG_m) \lra D\left(\formalgroup^{\un}\right),  
$$
where we remark that the Cartier dual of $\formalgroup^{\un}$ in this sense agrees with our usage of the notation $D(\formalgroup^{\un})$.\footnote{Indeed, the Cartier duality assignment of \cite{lurie2017elliptic} restricts to the duality equivalence of Proposition~\ref{spectralduality}} Note further that there is a canonical map of monoid objects 
$$
\Z \lra D(\GG_m) = \hom_{\on{Fun}(\on{CAlg}_{\Sph_p}^{\cn},\on{CMon}^{\gp})}(\GG_m, \GG_m); 
$$
this can be seen at the level of functors of points. The composite map 
$$
\Z \lra D(\GG_m) \lra D\left(\formalgroup^{\un}\right)
$$ 
gives our candidate lift. Base changing this along $\Spec(\Z_p) \to \Spec (\Sph_p)$, we obtain a map  $\Z  \simeq D(\GG_m )\to \mathsf{Fix}$, which is a map of ring schemes over $\Z_p$. This follows by construction since  on $A$-points, it is the map induced by the universal property of formal completion: 
\begin{align*}
\Hom_{\on{coGroup}(\on{CAlg})}(A[t, t^{-1}], A[t, t^{-1}])  \simeq \hom_{\Grp}(\GG_m, \GG_m)(A) 
 & \lra \hom_{\Grp}\left(\widehat{\GG_m}, \GG_m\right)(A) \\ 
  & \hphantom{\lra} \simeq \hom_{\Grp}\left(\widehat{\GG_m}, \widehat{\GG_m}\right)(A) \\ 
 & \hphantom{\lra}\simeq  \Hom_{\on{coGroup}(\on{Adic})}(A[[t]], A[[t]]),
\end{align*}
which is a map of rings.  Note that the equivalence on the third line arises from the fact that every group map $\widehat{\GG_m} \to \GG_m$ factors as $\widehat{\GG_m} \xrightarrow{ \alpha} \widehat{\GG_m} \xrightarrow{ i} \GG_m$. Hence we recover the map $\Z \to \mathsf{Fix}$ of ring schemes. 
By taking classifying stacks of this lift, we obtain the desired map
\begin{equation*}\pushQED{\qed}
S^1 \simeq B \Z \lra B D\left(\formalgroup^{\un}\right).\qedhere
\popQED
	\end{equation*}
\renewcommand{\qed}{}    
\end{proof}

Let $X= \Spec(A)$ be an affine spectral scheme. By taking mapping spaces, Proposition~\ref{Prop:9.8} furnishes a map
$$
\Map\left( B \formalgroup_{\un}^{\vee} , X\right)  \lra \Map\left(S^1 ,X\right);
$$
applying global sections further  begets a map
$$
f\colon \on{THH}(A) \lra R \Gamma\left(\Map\left( B \formalgroup_{\un}^{\vee} , X\right), \OO \right) 
$$
of $E_{\infty}$ $\Sph_{p}$-algebras.

\begin{thm} \label{Thm:9.9}
Let   
$$
f_p\colon \on{THH}\left(A; \mathbb{Z}_p\right) \lra R \Gamma\left(\Map\left( B \formalgroup_{\un}^{\vee} , X\right), \OO \right)^{\widehat{ }}_p
$$
denote the $p$-completion of the above map. Then $f$ is an equivalence. 
\end{thm}

\begin{proof}
Since this is a map of $p$-complete spectra, it is enough to verify that it is an equivalence upon tensoring with the Moore spectrum $\Sph_p / p$. In fact, since these are both connective spectra, one can go further and test this simply by tensoring with $\mathbb{F}_p$ (\textit{e.g.}, by \cite[Corollary A.33]{mao2020perfectoid}). Hence we are reduced to showing that 
$$
\on{THH}\left(A; \mathbb{Z}_p\right) \otimes_{\Sph_p} \mathbb{F}_p \lra R \Gamma\left(\Map\left( B \formalgroup_{\un}^{\vee} , X\right), \OO \right)^{\widehat{ }}_p \otimes \mathbb{F}_p 
$$
is an equivalence of $E_\infty$ $\mathbb{F}_p$-algebras. 
By generalities on topological Hochschild homology, we have the following identification of the left-hand side: 
$$
\on{THH}\left(A; \mathbb{Z}_p\right) \otimes_{\Sph_p} \mathbb{F}_p  \simeq \on{HH}\left(A \otimes_{\Sph_p} \mathbb{F}_p / \mathbb{F}_p \right). 
$$
Now we can use Theorem~\ref{representability} to identify the right-hand side with the global sections  of the  mapping stack  
$$
\Map\left(B \formalgroup_{\un}^{\vee}, X\right)  \times \Spec\left(\mathbb{F}_p\right) \simeq  \Map\left(B \formalgroup_{\un}^{\vee} \times \Spec \left(\mathbb{F}_p\right), X \times \Spec \left(\mathbb{F}_p\right)\right). 
$$
By Proposition~\ref{hochschildhomologygm}, this is precisely $\on{HH}(A \otimes_{\Sph_p} \mathbb{F}_p / \mathbb{F}_p )$, whence the equivalence.
\end{proof}

\section{Filtrations in the spectral setting}\label{sec10}
In Section~\ref{deformationofGm}, an interpretation of the HKR filtration on Hochschild homology was given in terms of a degeneration of $\widehat{\GG_m}$ to $\widehat{\GG_a}$. Moreover, this was expressed as an example of the  deformation to the normal cone construction of Section~\ref{deformationsection}. 

In Section~\ref{sec9}, we further saw that these $\formalgroup$-Hochschild homology theories may be lifted beyond the integral setting. A natural question then arises: do the filtrations come along for the ride as well? Namely, does there exist a filtration on $\on{THH}^{\formalgroup}(-)$ which recovers the filtered object corresponding to $\on{HH}^{\formalgroup}(-)$ upon base changing along $R^{\un}_{\formalgroup} \to k$?

We will not seek to answer this question here. However, we do give a reason why some negative results might be expected. As mentioned in the introduction, many of the constructions do work integrally. For example, one can talk about the deformation to the normal cone $\Def_{\filstack}(\formalgroup)$ of an arbitrary formal group over $\Spec(\Z)$. If we apply this to $\widehat{\GG_m}$, we obtain a degeneration of the formal multiplicative group to the formal additive group. We let $\Def_{\filstack}(\widehat{\GG_m})^{\vee}$ be the Cartier dual, as in Section~\ref{filteredformalgroupsection}.  In \cite{scheme}, the Cartier dual to $\widehat{\GG_m}$ is described to be $\Spec(\on{Int}(\Z))$, the spectrum of the ring of integer-valued polynomials on $\Z$. Moreover, it is shown that $ B \Spec(\on{Int}(\Z))$ is the affinization of $S^1$; hence one can recover (integral) Hochshild homology from this. 

Let us suppose there exists a lift of $\Def(\widehat{\GG_m})^{\vee}$ to the sphere spectrum, which we will denote by  $\Def^{\Sph}(\widehat{\GG_m})^{\vee}$. This would allow us to define a mapping stack in the $\infty$-category $\sStk_{\filstack}$ of spectral stacks over the spectral variant of $\filstack$. By the results of \cite{geometryofilt}, this comes equipped with a filtration on its cohomology, which we would like to think of as recovering topological Hochschild homology.  

However, over the special fiber $B \GG_m \to \filstack$, we would expect that such a lift $\Def^{\Sph}(\widehat{\GG_m})^{\vee}$ recovers the formal additive group $\widehat{\GG_a}$. More precisely, we would get a formal group over the sphere spectrum $\formalgroup \to \Spec(\Sph)$ which pulls back to the formal additive group $\GG_a$ along the map $\Sph \to \Z$.  However, by \cite[Proposition 1.6.20]{ellipticII}, this cannot happen. Indeed, there it is shown that $\widehat{\GG_a}$ does not belong to the essential image of $\on{FGroup}(\Sph) \to \on{FGroup}(\Z)$. 

We summarize this discussion into the following proposition. 

\begin{prop} \label{negativeresult}
There exists no lift of\, $\Def_{\filstack}(\formalgroup_m)$ over to the sphere spectrum. In particular, there exists no formal group $\widetilde{\formalgroup}$ over $\filstack$ relative to $\Sph$ such that $\widetilde{\formalgroup} \times \Spec(\Z) \simeq \Def_{\filstack}(\widehat{\GG_m})$. 
\end{prop}


\providecommand{\bysame}{\leavevmode\hbox to3em{\hrulefill}\thinspace}
\providecommand{\MR}{\relax\ifhmode\unskip\space\fi MR }
\providecommand{\MRhref}[2]{%
  \href{http://www.ams.org/mathscinet-getitem?mr=#1}{#2}
}
\providecommand{\href}[2]{#2}

\end{document}